%% file: InvolutiveMovingFrames.tex
\documentclass[12pt]{article}

\usepackage{setspace}
\usepackage{amsmath,amscd}
\usepackage{amsfonts}
\usepackage{verbatim}
\usepackage{amsthm, thmtools}
\usepackage{mathtools}
\usepackage[english]{babel}
\usepackage{amssymb}
\usepackage[T1]{fontenc}
\usepackage{enumerate}
\usepackage{lmodern}
\usepackage[shortlabels]{enumitem}
\usepackage[a4paper,bindingoffset=0.2in,left=1in,right=1in,top=1in,bottom=1in,footskip=.6in]{geometry}
\usepackage{blindtext}
\usepackage{leftidx}
\usepackage{mathrsfs}
\usepackage[linesnumbered,ruled]{algorithm2e}
\usepackage{setspace}
\usepackage{tikz}

\usepackage{textcomp}
\usetikzlibrary{positioning}
\usetikzlibrary{matrix}
\usetikzlibrary{arrows}
\usetikzlibrary{calc}
\usepackage{anyfontsize}

\usepackage{enumerate}
\usepackage{fullpage}
\usepackage{titlesec}
\titleformat*{\section}{\large\bfseries}
\usepackage[autostyle]{csquotes} 
\usepackage{bbm}
\usepackage{bm}

\usepackage{xypic}
\input xy

\numberwithin{equation}{section}

\newtheorem{theorem}{Theorem}[section]
\newtheorem*{theorem*}{Theorem}
\newtheorem{lemma}[theorem]{Lemma}

\newtheorem{corollary}[theorem]{Corollary}

\theoremstyle{definition}
\newtheorem{definition}[theorem]{Definition}
\newtheorem{Ex}[theorem]{Example}
\newtheorem{Rem}[theorem]{Remark}

  \makeatletter
\renewcommand*{\@cite@ofmt}{\bfseries\hbox}
\makeatother

\include{macrosLegit}

\title{\normalsize\bf Involutive moving frames}
\author{\normalsize\OO rn Arnaldsson}
\date{} 

\begin{document}
\maketitle


\input{imf-Introduction_forOrn}

\input{ex-EL}

\input{imf-PseudoGroups}

\input{imf-Core}

\input{imf-Examples2}

\newpage






\newpage
\bibliographystyle{plain}

\bibliography{./InvolutiveMovingFrames.bbl}

\end{document}

%% file: macrosLegit.tex
\newcommand{\ox}{\omega^x}
\newcommand{\ou}{\omega^u}
\newcommand{\op}{\omega^p}
\newcommand{\oi}{\omega^i}

\renewcommand{\o}{\omega}

\newcommand{\Pe}{\Phi_\varepsilon}
\newcommand{\vb}{\mathbf{v}}

\newcommand{\vbinf}{\widehat{\mathbf{v}}^\infty}
\newcommand{\Vbinf}{\widehat{\mathbf{V}}^\infty}

\newcommand{\Jinf}{J^{\infty}}
\newcommand{\Jinfe}{J^\infty(\E)}

\newcommand{\Gt}{\widetilde{\G}}

\newcommand{\pbg}{\widetilde{\G}_\infty}

\newcommand{\zinf}{z^{(\infty)}}
\newcommand{\zxinf}{z\restrict{x}^{(\infty)}}

\renewcommand{\Xi}{X^i}

\newcommand{\XiK}{X^i_K}

\newcommand{\uaj}{u^\alpha_j}

\newcommand{\uaJ}{u^\alpha_J}

\newcommand{\UaJ}{U^\alpha_J}
\newcommand{\UaJi}{U^\alpha_{J,i}}
\newcommand{\uaJi}{u^\alpha_{J,i}}

\newcommand{\ub}{\bar{u}}

\newcommand{\Ua}{U^\alpha}

\newcommand{\mui}{\mu^i}

\newcommand{\muiK}{\mu^i_K}

\newcommand{\muiKj}{\mu^i_{K,j}}

\newcommand{\upiK}{\Upsilon^i_K}
\newcommand{\upjJ}{\Upsilon^j_J}

\renewcommand\aa{\"a}

\newcommand\OO{\"O}

\newcommand{\jpx}{ j^\infty\varphi\restrict{x}}
\newcommand\jinf{j^\infty}

\newcommand{\resid}{\restrictbig{\widetilde{\mathbbm{1}}}}

\newcommand{\resone}{\restrictbig{\one}}

\newcommand{\intc}{\int\!\!}

\renewcommand{\sb}{\bar{s}}

\newcommand\taut{\tilde{\tau}}

\xyoption{all}

\newcommand{\Sb}{\bar{S}}

\renewcommand{\th}{^{\text{th}}}

\newcommand{\xb}{\bar{x}}

\newcommand{\p}{\varphi}
\newcommand{\ep}{\varepsilon}

\newcommand{\Nn}{\mathbb{N}_0^n}
\newcommand{\Nz}{\mathbb{N}_0}

\newcommand{\B}{\mathcal{B}}
\newcommand{\Bu}{\mathcal{B}^u}

\newcommand{\Bub}{\mathcal{B}^{\ub}}

\newcommand{\Bt}{\widetilde{\mathcal{B}}}

\newcommand{\A}{\mathcal{A}}
\newcommand{\K}{\mathcal{K}}

\newcommand{\E}{\mathcal{E}}
\newcommand{\D}{\mathcal{D}}
\newcommand{\X}{\mathcal{X}}
\newcommand{\U}{\mathcal{U}}
\newcommand{\G}{\mathcal{G}}
\renewcommand{\S}{\mathcal{S}}

\renewcommand{\H}{\mathcal{H}}

\newcommand{\Rr}{\mathcal{R}}

\newcommand{\R}{\mathbb{R}}

\newcommand{\dis}{\displaystyle}
\newcommand{\mb}{\mathbf{m}}
\newcommand{\ib}{\mathbf{i}}

\newcommand{\resx}{\restrict{x}}

\newcommand{\resz}{\restrict{z}}

\newcommand\one{\mathbbm{1}}

\newcommand{\bpm}{\begin{pmatrix}}
\newcommand{\epm}{\end{pmatrix}}
\newcommand{\bbm}{\begin{bmatrix}}
\newcommand{\ebm}{\end{bmatrix}}
\newcommand{\beq}{\begin{equation}}
\newcommand{\eeq}{\end{equation}}
\newcommand{\bex}{\begin{Ex}}
\newcommand{\eex}{\end{Ex}}
\newcommand{\bre}{\begin{Rem}}

\newcommand{\qv}{^{(q)}}

\newcommand\qb{\bar{q}}

\newcommand\av{^\alpha}

\def\restrict#1{\raise-.3ex\hbox{\scriptsize\ensuremath|}_{#1}}
\def\restrictbig#1{\raise-.3ex\hbox{\large\ensuremath|}_{#1}}
\def\irestrict#1{\raise+.7ex\hbox{\scriptsize\ensuremath|}^{#1}}

\def\comp{\raise 1pt \hbox{$\,\scriptstyle\circ\,$}}

\def\interior{\mathbin{\hbox{\hbox{{\vbox
    {\hrule height.4pt width6pt depth0pt}}}\vrule height6pt width.4pt depth0pt}\,}}
    
\newcommand{\pder}[2][]{\frac{\partial#1}{\partial#2}}

%% file: imf-Introduction_forOrn.tex
\begin{abstract} 
By combining the ideas of Cartan's equivalence method and the
method of the equivariant moving frame for pseudo-groups, we develop an
efficient method for solving equivalence problems arising from 
horizontal Lie pseudo-group actions. The key is a pseudo-group analog of the
classic result that characterizes congruence of submanifolds in Lie groups in
terms of equivalence of the Lie group's Maurer-Cartan forms. This result, when
combined with the fundamental recurrence formulas for the moving frame for
pseudo-groups, will allow for a hybrid equivalence method that will extend and
illuminate its two progenitors. Furthermore, incorporating the recurrence formula from the equivariant moving frame calculus provides great computational simplifications in solving equivalence problems.  

We apply the method to substantial equivalence problems
such as (point) equivalence of second order ordinary differential equations, (point) divergence
equivalence of Lagrangians on the line and equivalence of linear second order
differential operators. These examples will demonstrate the ease with which the combined
equivalence method can solve such problems. 
\end{abstract}

\begin{section}{Introduction}

A geometric structure on a manifold is preserved when going from one coordinate
system to another. \'Elie Cartan, \cite{Car53}, encoded the conditions on a
change of coordinate transformation, for many different geometric structures, as
an equivalence of $G$-structures. An O$(n)$-structure, for example, describes a
Riemannian metric. Cartan's equivalence method computes the local invariants (or
curvatures) of the geometry at hand, and, in the case of an O$(n)$-structure, leads
one to the Levi-Civita connection. 
A bridge from the local
invariants to global invariants is the Chern-Weil theory of characteristic
classes, \cite{chern52}.
%
%
%
%
 After the initial successes of Cartan and his disciples, Cartan's equivalence
 method lay dormant for decades until the 1980s when its importance in solving
 a variety of equivalence problems in differential equations (\cite{Bryant95co, Kamran89,
Kamran87}), calculus of variations (\cite{Bryant87,  Gardner85, Hsu89,
Kamran91}), control theory (\cite{Gardner83}) and classical invariant theory
(\cite{olver90, Olver99}) became clear. However, the complexity of the
calculations involved in solving an equivalence problem using the method has
 rendered it a largely theoretical tool, even with the aid of modern computer algebra
systems. 

In a series of papers
\cite{olver05, olver08moving, olver09}, Olver and Pohjanpelto extended the
equivariant moving frame for Lie groups developed in \cite{Fels99} to the
infinite-dimensional realm of a Lie pseudo-group, $\G$, of local diffeomorphisms of a
manifold, $\X$, and provided a practical structure theory of these infinite
dimensional analogs of Lie groups. In local coordinates, $\R^n$, on $\X$, Lie
pseudo-groups are determined by systems of differential equations $\G_q\subset
J^q(\R^n\times\R^n\to\R^n)$, $q\leq\infty$, in jet space. Olver and
Pohjanpelto's approach relied on the fact that the subsets $\G_q$ carry, as
first emphasized by Ehresmann, \cite{Ehresmann53}, a \emph{groupoid structure}.
They then found a remarkable basis for the contact structure of the spaces
$\G_q$ that is \emph{invariant} under the groupoid operation. Naturally, these
contact forms were called the Maurer-Cartan forms of $\G$. Furthermore, in
\cite{olver05}, an algorithmic way of obtaining the \emph{structure equations}
of the Maurer-Cartan forms was found. The equivariant moving frames for Lie groups as well as Lie pseudo-groups have found a plethora of applications, including classical invariant theory (\cite{berchenko00, Olver99}), object recognition and symmetry detection (\cite{boutin00, calabi98}), invariant finite difference numerical schemes (\cite{calabi96, olver01}), invariant Euler-Lagrange equations (\cite{kogan03}) and geometric flows (\cite{kenney09, olver08}). However, applications of the equivariant moving frame for Lie pseudo-groups have mostly relied on the pseudo-group action eventually becoming free (in a specific sense). Indeed, the main result of \cite{olver09} says that if the pseudo-group action becomes free at some jet order, then it remains free. Cartan's equvalence method requires no such restriction and it has been of interest to extend the pseudo-group moving frame beyond free actions, which is one of the contributions of this paper. In fact we combine Cartan's method with the equivariant moving frame
for pseudo-groups, \cite{olver08moving}, and the structure theory for Lie
pseudo-groups in \cite{olver05}, to obtain a powerful hybrid equivalence method
for a large class of equivalence problems. This hybrid method utilizes both the
geometry of Cartan's differential forms and the recurrence formula of the
equivariant moving frame and as such can dramatically reduce the computational load of the type of
low dimensional concrete applications mentioned above.

%

We shall study the following, quite general, question.
Given a Lie pseudo-group $\G$ of transformations on $\X$ whose action induces an
action on a fiber bundle $\E\to\X$ (that typically depends on the higher jets of
the elements of $\G$), when are two sections of $\E$ locally congruent under an
element $\p\in\G$? We shall say that $\G$ acts \emph{horizontally} on $\E$ in
this situation. The equivalence of two sections of a tensor bundle over a
manifold $M$ under a general change of coordinates is one example of this
problem since a change of variable, $\p$, induces a transformation on the
components of the tensor fields that depends on the 1-jets, $j^1\p$, of $\p$.
Equivalence of Lagrangians under a change of variables is another, see Section
\ref{sec:basicobj}. We should mention that due to reliance on local solvability
of the differential systems describing Lie pseudo-groups, in general we must
restrict to the real-analytic category. This paper solves this problem in a novel and powerful way that generalizes Cartan's solution to the finite dimenional case of Lie groups. Underlying the method of Cartan's Rep\`ere Mobile, \cite{car37}, and the
equivariant moving frame is the fact that two submanifolds, $i:S\to G$ and
$\bar{i}:\Sb\to G$, in a Lie group $G$ are congruent under an element $g\in G$,
\[ g\cdot i(S)=\bar{i}(S), \] if and only if there is a map $\psi:S\to \Sb$
preserving the (left-invariant) Maurer-Cartan forms on $G$ pulled-back to $S$
and $\Sb$. See, e.g., \cite{griffiths74} for a proof. Given a Lie group $G$ acting
effectively (and regularly) on a manifold $M$, this result along with the
equivariant moving frame allows for the complete solution of the congruence
of submanifolds of $M$ under the action of $G$. Equivalence problems
under Lie pseudo-group actions are much more difficult and Cartan developed his
rather complicated equivalence method to handle them. Based on the newly
discovered Maurer-Cartan forms for Lie pseudo-groups we will develop an approach
to congruence problems for Lie pseudo-groups (in the horizontal case) that is
analogous to the finite dimensional case. The key step is the following
generalization of the above result to the infinite dimensional realm (see
Theorem \ref{Gco_gen}),

\begin{theorem*} Let $s$ and $\sb$ be two sections of $\G_p\to\X$ and let $f$ be
a local diffeomorphism on $\X$ such that \beq\label{thmeq} \begin{aligned}
f^*\sb^*\muiK&=s^*\muiK,\quad |K|<p,\quad\text{and}\\ f^*\sb^*\tau&=s^*\tau.
\end{aligned} \eeq Then $f\in\G$ and $R_f\cdot s=\sb$.  \end{theorem*}

Here, $\muiK$ are the Maurer-Cartan forms of $\G$, $R_f\cdot s$ denotes the
groupoid product $s\cdot j^pf^{-1}$ in $\G_p$ and $\tau$ is the target map in
the groupoid $\G_p$. (The product $s\cdot j^pf^{-1}$ is only defined at points
where the source of $s$ and target of $f^{-1}$ agree.)

Before moving on to the consequences of the above theorem we must establish some
notation from the theory of the equivariant moving frame for pseudo-groups.
For the two infinite bundles over the base manifold $\X$, $\G_\infty\to\X$ and
$\Jinf(\E)\to\X$, denote the elements of $\G_\infty$ by $\jinf\p\resx$ and
those of $\Jinf(\E)$ by $\zinf\resx$. The pull-back bundle
$\Gt_\infty\to\Jinf(\E)$, coordinatized by $(\jinf\p\resx, \zinf\resx)$ has a
double fibriation,

\[ \begin{tikzpicture} \node (A0) at (0,0) {$\Gt_\infty$}; \node (A1) at
(240:2cm) {$\Jinf(\E)$}; \node (A2) at (300:2cm) {$\Jinf(\E)$,};
\draw[->,font=\scriptsize] (A0) edge node[left] {$\tilde{\sigma}$} (A1) (A0)
edge node[right] {$\tilde{\tau}$} (A2);

\end{tikzpicture} \] where $\tilde{\sigma}(\jinf\p\resx, \zinf\resx)=\zinf\resx$
and $\tilde{\tau}(\jinf\p\resx, \zinf\resx)=\jinf\p\resx\cdot \zinf\resx$ is the
prolonged action of $\G$ on $\Jinf(\E)$. The elements of $\G$ act on the
pull-back bundle $\Gt_\infty$ by \emph{right-regularization}; for $\psi\in\G$ we
define \[ R_\psi\cdot
(\jinf\p\resx,\zinf\resx)=(\jinf\p\resx\cdot\jinf\psi^{-1}\restrict{\psi(x)},
\jinf\psi\resx\cdot\zinf\resx).  \] The map $\tilde{\tau}$ is invariant under
this action, as can easily be checked, and the components of $\tilde{\tau}$,
which are complicated expressions in the jet coordinates on $\G_\infty$
and $\Jinf(\E)$, are called \emph{lifted invariants}. A common level set of (an
arbitrary number of) lifted invariants is called a \emph{partial moving frame}.
Note that a partial moving frame is an invariant set under the regularized
action. 

For local coordinates $(x,u)$ on $\E$, where $x\in\R^n$ and $u\in\R^m$, we write
the jets $\zinf\resx\in\Jinf(\E)$ as \[ \zinf\resx=(x,u,\ldots, \uaJ,
\ldots),\quad 1\leq\alpha\leq m, J\in\Nz^n, \] and the components of
$\tilde{\tau}$ with capitalized letters, \[ \tilde{\tau}(\jinf\p\resx,
\zinf\resx)=(X,U,\ldots, \UaJ, \ldots).  \] Now, the congruence problem of
sections, $\jinf u$ and $\jinf\ub$ of $\Jinf(\E)$ under the action of $\G$ can
be rewritten as the problem of finding two congruent sections of $\G_p\to\X$
which are confined to the common level set $X=U=0$ (Corollary \ref{zfundcor}).
Theorem \ref{Gco_gen} now reduces this problem further to solving for $f$ in
(\ref{thmeq}). It is here that Cartan's $G$-structures emerge (see Section
\ref{partmaur}).

But since we simply recover Cartan's $G$-structures, what exactly has been
gained by these maneuvers? First, the above provides an algorithmic way of
writing down a $G$-structure for an arbitrary equivalence problem (for
horizontal actions). Second, and much more importantly, this approach invites
the \emph{recurrence formula} from the equivariant moving frame,
\cite{olver08moving}, into the picture as well as the structure equations for
the Maurer-Cartan forms found in \cite{olver05}. We have reduced the congruence
problem to solving for $f$ in (\ref{thmeq}). This is an exterior differential
system whose solution will depend on the analysis of the exterior derivatives of
the Maurer-Cartan forms, $d\muiK$, \emph{restricted to a partial moving frame},
and it is here that the recurrence formula and structure equations emerge.

The structure equations for $d\muiK$ are the restrictions to $\G_p$ of the
equations, found in \cite{olver05},\beq\label{intmuik} d\muiK=\sum_{1\leq j\leq
n}\o^j\wedge\muiKj+\sum_{\substack{L+M=K \\ |M|\geq1}}\binom{K}{L}\sum_{1\leq
j\leq n}\mu^i_{L,j}\wedge\mu^j_M.  \eeq Furthermore, the recurrence formula
tells us that 
\[ 
d\UaJ
=
\UaJi\oi+\sum\lambda\left(\vb^\infty(\uaj)\right)\muiK
=
\UaJi\oi+\sum\lambda\left(\frac{\partial
\UaJ}{\partial\XiK}\resone\right)\muiK.  
\] 
(In the above equations, $\oi$ are
certain invariant horizontal forms (i.e., 1-forms on the base manifolds with
coefficients that are real-analytic functions on $\Gt_\infty$), $\lambda$ is
an operator designed to \emph{invariantize} objects on $\Jinf(\E)$, see Section
\ref{sec:lift} and $\vb^\infty$ is the infinitesimal generator of $\G$ on $\Jinf(\E)$.) For example, if $U=0$ on our partial moving frame, then the
recurrence formulas imply that \beq\label{intua}
0=d\Ua=\Ua_i\oi+\sum\lambda\left(\frac{\partial
\Ua}{\partial\XiK}\resone\right)\muiK.  \eeq Hence, when combined with the
structure equations (\ref{intmuik}), we obtain the complete structure of the
Maurer-Cartan forms on the partial moving frame. 

As mentioned previously the Maurer-Cartan forms $\muiK$, when restricted to
partial moving frames, form the classic notion of a $G$-structure for the
equivalence problem. For example, if $\G$ is determined by first order equations
$\G_1$ the invariant horizontal forms $\o^1, \ldots, \o^n$ provide the
$G$-structure. Cartan's \emph{structure equations} for this $G$-structure are
the expressions for $d\oi$ (see \cite{Gardner89, olver95} for expositions
of the method). But in our formulation we obtain these very easily as follows.
First of all, by the general structure equations for Lie pseudo-groups
(\ref{intmuik}) we have \beq\label{intdo} d\oi=\o^j\wedge\mui_j \eeq Restricting
these equations to $\G_1$ will give certain linear dependencies among the first
order Maurer-Cartan forms $\mui_j$. Futhermore, (\ref{intua}) gives further
dependencies among the Maurer-Cartan forms. Choosing principal and parametric
derivatives and plugging them into (\ref{intdo}) gives expressions of the form
\[ d\oi=\cdots+\Ua_k\o^k\wedge\o^j, \] where $\cdots$ are forms involving first
order Maurer-Cartan forms $\mui_j$. The first order lifted invariants $\Ua_k$
that emerge at this stage as the coefficients of the purely horizontal two-forms
are Cartan's \emph{torsion coefficients} and we normalize them to convenient
constants to obtain a higher order partial moving frame. Again, the recurrence
formula tells us that once we normalize $\Ua_k$ we have introduced some further
dependencies among the invariant forms, 
\[
0=d\Ua_k=\Ua_{kl}\o^l+\sum_{|K|\leq2}\lambda\left(\frac{\partial
\Ua_k}{\partial\XiK}\resone\right)\muiK.  
\] 
The key simplifying
feature of the recurrence formula is that the coefficients of $\muiK$ in
(\ref{intua}) are being evaluated at the identity section and so are rather
simple, easily computable expressions (and even more so by software such as
\textsc{Mathematica}). In Section \ref{examples1} we solve some
difficult equivalence problems where these simplifications are evident.

As in Cartan's equivalence method, we stop this process of exterior
differentiation of the invariant forms on the partial moving frames, and
normalization of lifted invariants, when one of two things happen. Either when
the top order forms satisfy \emph{Cartan's test for involution} or when we
manage to solve for \emph{all} top order forms in the recurrence formulas. In
either case we say we have obtained an \emph{involutive moving frame}. Checking for involution is made simple by already knowing all the structure equations (\ref{intmuik}). For a proof that the process terminates, see \cite{IMF}.

This paper is structured as follows. We begin by giving a worked example that demonstrates the workings of our equivalence method so the reader can immediately see its many benefits and will hopefully be motivated to study the details of the paper. The following Section \ref{pseudogroups} gives an
overview of the groupoid approach to Lie pseudo-groups and sets
the stage for our method of involutive moving frames in the special case of
horizontal actions. Section \ref{core} forms the core of the paper. It first gives a quick review of the equivariant moving frame and 
gives a novel definition of a partial moving frame, \cite{olver11},
in the ``groupoid spirit.'' This definition allows us to apply our equivalence
method to so-called singular jets that have, so far, been beyond the reach of
the equivariant moving frame (while within that of Cartan's method). Section \ref{SGp} contains the fundamental proof
of equivalence of sections in the groupoids associated with a Lie pseudo-group.
The following section shows how this theorem provides some classical
$G$-structures of Cartan, and connects Cartan's equivalence method with (partial)
moving frames. Finally, in
Section \ref{examples1} we solve some significant equivalence problems to
demonstrate the power of this method of involutive moving frames. These include (point)
equivalence of second order ordinary differential equations and equivalence of second order
differential operators. 

\end{section}

%% file: ex-EL.tex
\begin{section}{Divergence equivalence of Lagrangians}

We begin this paper by giving a worked example of a difficult equivalence problem, that of Lagrangians $\dis \int L(x,u,u_x)dx$ under point transformations modulo total divergence (see Example \ref{divel} for the set-up). The pseudo-group of point transformations (writing $p=u_x$), $(x,u,p)\mapsto (X,U,P)$, $\G$, has defining equations
\beq\label{eq:point}
X_p=U_p=0,\quad U_x=P(X_x+pX_u)-pU_u.
\eeq
We then extend the point-transformation $(x,u,p)\mapsto (X,U,P)$ to a two dimensional space of variables representing $L_{pp}$ and $\tilde{E}(L)=L_u-L_{px}-pL_{pu}$. Call these variables $w$ and $z$, respectively. The $w$ and $z$ transform according to (see Example \ref{divel})
\begin{align*}
{}&w\mapsto W=\frac{w}{P_p^2\cdot(pX_u+X_x)},\\
&z\mapsto Z=\frac{z}{P_p\cdot(pX_u+X_x)^2}+\frac{pP_u+P_x}{pX_u+X_x}\cdot W,
\end{align*}
and two sections of the trivial bundle $(x,u,p,z,w)\mapsto(x,u,p)$ are equivalent under this extended action if and only if the associated Lagrangians are divergence equivalent. 

Differentiating the last equation in (\ref{eq:point}) w.r.t. $p$ gives the integrability condition
\[
P_p=\frac{U_u-P X_u}{p X_u+X_x}.
\]
This makes $\G_1$ formally integrable. Prolonging these equations once and linearizing gives the following linear dependencies among the Maurer-Cartan forms, up to second order (after setting $X=U=P=0$).
\beq\label{elstr}
\begin{aligned}
{}&\mu^p_p=\mu^u_u-\mu^x_x,~~\mu^u_x=-\o^p,~~\mu^p_{pp}= -2\mu^x_u,~~\mu^p_{pu}=\mu^u_{uu}-\mu^x_{ux},\\
{}&\mu^p_{px}=-\mu^x_{xx}+\mu^p_u,~~\mu^u_{ux}=\mu^p_u,~~\mu^u_{xx}=\mu^p_x,
\end{aligned}
\eeq
as well as $\mu^x_{pi}=\mu^u_{pi}=0$, $i\in\{x,u,p\}$. We normalize $W$ to 1 and $Z$ to 0, and the recurrence formula gives, using (\ref{elstr}), that
\begin{align*}
{}&0=dW=W_i\oi-2\mu^u_u+\mu^x_x,\\
&0=dZ=Z_i\oi+\mu^p_x.
\end{align*}
Computing all $d\oi$ gives
\beq\label{eldo}
\begin{aligned}
{}&d\o^x=\o^x\wedge\mu^x_x+\o^u\wedge\mu^x_u=\o^x\wedge(2\mu^u_u-W_i\oi)+\o^u\wedge\mu^x_u,\\
&d\o^u=-\o^x\wedge\o^p+\o^u\wedge\mu^u_u,\\
&d\o^p=\o^x\wedge(-Z_i\oi)+\o^u\wedge\mu^p_u+\o^p\wedge(W_i\oi-\mu^u_u).
\end{aligned}
\eeq
We can see that the purely horizontal parts involve the $W_i$ and $Z_i$ and so we compute the $d_G$ of these using the recurrence formula. 

\begin{Rem}\label{rem:zp}
Notice that
\[
z_p=-w_x-pw_u
\]
and so the jets of $w$ and $z$ are not entirely functionally independent. We therefore always replace a $z$ jet that involves a $p$ derivative with an expression involving only $w$ jets.
\end{Rem}

We find that we can normalize $X_u$ from $W_P=0$ to obtain
\[
0=W_{Pi}\oi+3\mu^x_u,\quad\text{and}\quad X_u=-\frac{U_u^2w_p}{3w^2}.
\]
We also find (notice that we skip $Z_P$, cf. Remark \ref{rem:zp})
\begin{align*}
d_GW_X&=\mu^x_{xx}-2\mu^p_u-2W_X\mu^u_u,\\
d_GW_U&=-2\mu^u_{uu}+\mu^x_{ux}-W_X\mu^x_u,\\
d_GZ_U&=\mu^p_{ux}+W_X\mu^p_u,\\
d_GZ_X&=\mu^p_{xx}.
\end{align*}
We can normalize $X_{xx}$, $U_{uu}$, $P_{ux}$ and $P_{xx}$ from these equations and we go back to our structure equations to find that
\beq\label{eldo2}
\begin{aligned}
{}&d\o^x=2\o^x\wedge\mu^u_u-\frac{1}{3}\o^u\wedge(W_{Pi}\oi),\\
&d\o^u=-\o^x\wedge\o^p+\o^u\wedge\mu^u_u,\\
&d\o^p=\o^u\wedge\mu^p_u-\o^p\wedge\mu^u_u.
\end{aligned}
\eeq
Now we must check the lifted invariants $W_{Pi}$, and we find that
\begin{align*}
d_GW_{PU}&=3\mu^x_{uu}-W_{PP}\mu^p_u,\\
d_GW_{PX}&=3 \mu^x_{ux}.
\end{align*}
while
\[
W_{PP}=\frac{U_u^2 \left(3 w w_{pp}-4
   w_p^2\right)}{3 w^4}.
\]
The first order group parameter $U_u$ can be normalized from $W_{PP}$ if and only if 
\[
3 w w_{pp}-4w_p^2\neq0 \quad\text{or, equivalently,}~3L_{pp}L_{pppp}-4L_{ppp}^2\neq0)
\]
and so the equivalence problem branches at this juncture. In case $3 w w_{pp}-4w_p^2=0$ we must check for involution of our $G$-structure provided by the $\oi$. The structure equations are
\beq\label{eldo3}
\begin{aligned}
d\o^x&=2\o^x\wedge\mu^u_u,\\
d\o^u&=-\o^x\wedge\o^p+\o^u\wedge\mu^u_u,\\
d\o^p&=\o^u\wedge\mu^p_u-\o^p\wedge\mu^u_u.
\end{aligned}
\eeq
and the first reduced Cartan character is the maximal rank of the set of one-forms
\[
\gamma_1\{\left(a\pder{\o^x}+b\pder{\o^u}+c\pder{\o^p}\right)\interior d\oi~|~i\in\{x,u,p\}\},
\]
where $\gamma_1$ projects onto the space of first order Maurer-Cartan forms. This set is
\[
\{2a\mu^u_u,~~ b\mu^u_u,~~ b\mu^p_u-c\mu^u_u\}
\]
and maximizing its rank is equivalent to maximize the rank of the matrix
\[
\bbm 2a & 0
\\
b & 0 \\
-c & b 
\ebm,
\]
which obviously gives the first reduced Cartan character $s^{(1)}_1=2$ and since this makes the matrix full rank the second and third reduced Cartan characters are zero, $s^{(1)}_2=s^{(1)}_3=0$. 

The only second order group parameter we have not managed to normalize is $P_{uu}$ and so Cartan's test for involution asks whether
\[
1=\#\{P_{uu}\}=\sum s^{(1)}_i\cdot i=2
\]
which is untrue and the the $G$-structure is not involutive and we must prolong to the next order. In our framework this just means that the collection of one-forms on our partial moving frame
\[
\{\o^x, \o^u, \o^p, \mu^u_u, \mu^p_u\}
\]
are invariant under equivalence maps. To obtain the structure equations for this system we must join $d\mu^u_u$ and $d\mu^p_u$ to (\ref{eldo3}). The Maurer-Cartan structure equations for Lie pseudo-groups already tell us what these are and we must simply restrict them to the partial moving frame. We find
\beq\label{eldmu}
\begin{aligned}
{}&d\mu^u_u=\oi\wedge\mu^u_{ui}+\mu^u_x\wedge\mu^x_u,\\
&d\mu^p_u=\oi\wedge\mu^p_{ui}+\mu^p_x\wedge\mu^x_u+\mu^p_u\wedge\mu^u_u+\mu^p_p\wedge\mu^p_u,
\end{aligned}
\eeq
and after plugging in all of our information from above, we have
\beq\label{eldmu2}
\begin{aligned}
{}&d\mu^u_u=\o^x\wedge\mu^p_u+\frac{1}{2}\o^u\wedge\left((W_{Ui}-\frac{1}{3}W_{PXi})\oi\right),\\
&d\mu^p_u=-\o^x\wedge(Z_{Ui}\oi)+\o^u\wedge\mu^p_{uu}+\o^p\wedge\left((\frac{1}{2}W_{Ui}+\frac{1}{3}W_{PXi})\oi\right).
\end{aligned}
\eeq
We must therefore compute $d_G$ of $Z_{Ui}$, $W_{PXi}$ and $W_{Ui}$.  We find
\beq
\begin{aligned}
d_GZ_{UU}&=\mu^p_{uux}-\mu^p_{uu},\\
d_GZ_{UX}&=\mu^p_{uxx}-Z_{PX}\mu^p_u+Z_{Ui}\oi-\frac{1}{3}W_{Pi}\oi\\
d_GW_{UX}&=\mu^x_{uxx}-2\mu^p_{uu}-3W_{UX}\mu^u_u,\\
d_GW_{UU}&=-2\mu^u_{uuu}+\mu^x_{uux},\\
d_GW_{PUX}&= 3 \mu^x_{uux}+W_{PXX}(W_{Pi}\oi)\\
d_GW_{PXX}&= 3 \mu^x_{uxx}-3W_{PXX}\mu^u_u,
\end{aligned}
\eeq
while $W_{PPX}=0$. These equations indicate that the first five of them can be normalized for a third order group parameter, but the third and last equations indicate that
\[
d_G(W_{PXX}-3W_{UX})=6\mu^p_{uu}+\dots
\]
and therefore $P_{uu}$ may be normalized from $W_{PXX}-3W_{UX}$. This is where our framework has an important computational benefit.

\begin{Rem}\label{rem:sim}
$W_{PXX}$ and $W_{UX}$ both depend on third order pseudo-group parameters and both are \underline{affine} in these. Meanwhile $W_{PXX}-3W_{UX}$ only depends on second order parameters. Therefore, when computing $W_{PXX}-3W_{UX}$ we can (since they must cancel anyway) eliminate every third order parameters that we come across. This may not seem like much of an advantage in this relatively simple example, but this observation is crucial in computationally heavier equivalence problems where our computer algebra software is asked to simplify enormous expressions; it help tremendously to get rid of all top order group parameters before simplifying.
\end{Rem}

Now, since we have normalized \emph{all} second order group parameters we have reduced the branch of  the equivalence problem where
\[
3 w w_{pp}-4w_p^2=0
\]
to an equivalence problem for the coframe
\[
\{\o^x, \o^u, \o^p, \mu^u_u, \mu^p_u\} 
\]
on the partial moving frame. The structure equations are
\beq\label{eldmu3}
\begin{aligned}
d\o^x&=2\o^x\wedge\mu^u_u,\\
d\o^u&=-\o^x\wedge\o^p+\o^u\wedge\mu^u_u,\\
d\o^p&=\o^u\wedge\mu^p_u-\o^p\wedge\mu^u_u,\\
d\mu^u_u&=\o^x\wedge\mu^p_u,\\
d\mu^p_u&=\o^u\wedge\mu^p_{uu}+2\mu^p_u\wedge\mu^u_u.
\end{aligned}
\eeq
From the normalization of $W_{PXX}-3W_{UX}$ and the recurrence formula we find that
\[
\mu^p_{uu}=\frac{1}{6}\left(W_{PXXi}-3W_{UXi}\right)\oi
\]
and so the final structure equations are
\beq\label{eldmu4}
\begin{aligned}
d\o^x&=2\o^x\wedge\mu^u_u,\\
d\o^u&=-\o^x\wedge\o^p+\o^u\wedge\mu^u_u,\\
d\o^p&=\o^u\wedge\mu^p_u-\o^p\wedge\mu^u_u,\\
d\mu^u_u&=\o^x\wedge\mu^p_u,\\
d\mu^p_u&=\frac{1}{6}\o^u\wedge\left(W_{PXXi}-3W_{UXi}\right)\oi+2\mu^p_u\wedge\mu^u_u.
\end{aligned}
\eeq
Now, the coefficient $W_{PPXX}-3W_{PUX}$ vanishes on our partial moving frame and so we only have $W_{PXXX}-3W_{UXX}$ left to compute. We first check the $d_G$ of this expression using the recurrence formula. We have
\[
d_G\left(W_{PXXX}-3W_{UXX}\right)=6\mu^p_{uux}-5W_{PXXX}\mu^u_u.
\]
We have already normalized $P_{uxx}$ from $Z_{UU}$ above, where we found
\[
d_GZ_{UU}=\mu^p_{uux}-\mu^p_{uu}
\]
and so to compute $W_{PXXX}-3W_{UXX}$ we may compute $W_{PXXX}-3W_{UXX}-6Z_{UU}$ where we can take advantage of Remark \ref{rem:sim} as the expression $W_{PXXX}-3W_{UXX}-6Z_{UU}$ will depend on (at most) first order pseudo-group parameters. This computation is not feasible to do by hand and so we leave it up to \textsc{Mathematica}. We find that on our partial moving frame $W_{PXXX}-3W_{UXX}$ has the form
\[
\frac{J}{3 w^3U_u^5}
\]
where $J$ is an enormous expression, which in this fonsize would cover a full page.

For the Lagrangian in $\dis \int\frac{1}{2}p^2dx$, $J$ vanishes (as well as $3ww_{pp}-4w_p^2$ and so we are in the above branch) and the contact transformations preserving this Lagrangian up to divergence form a five dimensional (local) Lie group whose Maurer-Cartan forms have structure equations
\beq\label{eldmu5}
\begin{aligned}
d\o^x&=2\o^x\wedge\mu^u_u,\\
d\o^u&=-\o^x\wedge\o^p+\o^u\wedge\mu^u_u,\\
d\o^p&=\o^u\wedge\mu^p_u-\o^p\wedge\mu^u_u,\\
d\mu^u_u&=\o^x\wedge\mu^p_u,\\
d\mu^p_u&=2\mu^p_u\wedge\mu^u_u.
\end{aligned}
\eeq

Better yet, all Lagrangians in this branch for which $J$ vanishes are equivalent and there is a five dimensional space of equivalence maps (in $\G$) between any two of them. We shall not continue this equivalence problem as its full analysis would take up too much space, but the above has demonstrated all of the key features of the combined equivalence problem established in this paper. 

Having sung the praises of the above routine, we must admit one shortcoming. The equivalence maps that our method computed must all come from the original pseudo-group $\G$ which acts on the space $\E$ of $(x,u,p,w,z)$ and the jet spaces $J^k(\E)$ by prolongation. This means that when solving equivalence problems for differential equations the method only computes the \emph{external symmetries}, while Cartan's method would also compute the \emph{internal symmetries}. (See \cite{anderson93} for a complete analysis of these matters.) Whether or not the moving frame can be utilized to compute internal symmetries of differential equations is an important open problem.

\end{section}

%% file: imf-PseudoGroups.tex
\begin{section}{Pseudo-groups}\label{pseudogroups}
This section gives a rapid overview of the groupoid approach to Lie pseudo-groups, their structure equations and the recurrence formula for lifted invariants, developed in the series of papers \cite{olver05, olver08moving, olver09}. We only prove the recurrence formula in the special case of horizontal actions, but the general case is identical, see \cite{olver08moving}. Since the theory of Lie pseudo-groups requires ``sufficiently regular'', formally integrable differential equations to be locally solvable, and since this is only true in general for analytic systems (by the Cartan-K\aa hler theorem), we shall assume real analyticity of all systems in this paper. Regularity is a thorny issue and a time-honored tradition in the formal theory of differential equations is to ``assume whatever regularity you need at each time''. Here, regularity is understood in the following specific sense. To describe this definition, let $\E\overset{\pi}{\to}\X$ be a bundle and let $\Rr_q$ be a $q\th$ order differential equation in the jet space $J^q(\E)$. Denote the canonical bundle maps between jet spaces by 
\[
J^p(\E)\overset{\pi^p_q}{\mathsmaller{\to}} J^q(\E),\quad \text{for}~p\geq q.
\]  
Further, denote the $t\th$ prolongation of $\Rr_q$ by $\Rr_{q,t}\subset J^{q+t}(\E)$ and set
\[
\Rr^{(s)}_{q,t}:=\pi^{q+t}_{q+t-s}(\Rr_{q,t}),\quad\text{for}~q+t\geq s.
\]
\begin{definition}\label{regularity}
The system $\Rr_q$ is \emph{regular}, with \emph{regularity order} $p^*$, if, for all $s,t$, $\Rr^{(s)}_{q,t}$ is a submanifold of $J^{q+t-s}(\E)$, and for all $s,t$ such that $q+t-s\geq p^*$ the reduced Cartan characters of $\Rr^{(s)}_{q,t}$ are constant along $\Rr^{(s)}_{q,t}$.
\end{definition}

The following example illuminates some of the issues this definition is meant to tackle.

\bex\label{ex:AbelEx}
Consider the first order, non-regular, formally integrable, differential system $\Rr_1$ for maps 
\[
(x,y,u)\mapsto(X(x,y,u), Y(x,y,u), U(x,y,u)), 
\]
for real valued functions $X,Y$ and $U$, determined by the equations
\[
X=x,\quad Y=y,\quad U=u+xU_x+yU_y.
\]
The symbol of $\Rr_1$, and all its prolongations, changes rather drastically close to $x=0$ and $y=0$ and on their intersection, the $u$-axis of the base manifold $\R^3$, it assumes a more degenerate form still as the determining equations decrease in order there. At all points in fibers above the $u$-axis the symbol module is all of $\R[\xi^1, \xi^2, \xi^3]\otimes\R^3$ and hence trivially involutive. However, the lack of regularity at these points prevents an application of the Cartan-K\aa hler theorem at these points (at all orders). Away from the $u$-axis the system is regular and, due to lack of integrability conditions, locally solvable.

Now adjoin the following second order equations to $\Rr_1$, to obtain the system $\Rr_2$,
\[
X_{ij}=Y_{ij}=U_{ij}=0,\quad\text{for all}~i,j\in\{x,y,u\}.
\]
It is easily seen that $\Rr_2$ is regular with regularity order 2 and hence locally solvable.
\eex

\begin{subsection}{Basic objects}\label{sec:basicobj}

Consider the jet bundle $\Jinf(\X\times\X\overset{\sigma}{\to}\X)$ for sections of the trivial bundle $\X\times\X\overset{\sigma}{\to}\X$ where $\X$ is an $n$-dimensional manifold. Let $\D(\X)$ denote the collection of all local diffeomorphisms of $\X$ and let $\D_\infty(\X)\subset\Jinf(\X\times\X\to\X)$ be the subbundle of all infinite jets of these. We shall drop the mention of $\X$ when it is clear what the base manifold is and simply write $\D$ and $\D_\infty$ instead of $\D(\X)$ and $\D_\infty(\X)$. Similarly, we denote by $\D_p(\X)$ the set of $p$-jets of transformations from $\D(\X)$. For local coordinates $x$ on $\X$ we have the induced jet coordinates $(x,X,\ldots, \XiK,\ldots)$ on $\D_\infty$ (and by truncation on $\D_p$). That is, for a local diffeomorphism $\p$, we have $\jpx=(x,X,\ldots, \XiK,\ldots)$, where $$\XiK=\frac{\partial^{|K|} \p}{\partial x^{K}}(x),\quad K\in\Nn.$$ 
 
The collection $\D$ forms a \emph{pseudo-group}, since if $\p\in\D$ then $\p^{-1}\in\D$ and the composition of two local diffeomorphisms is again a diffeomorphism \emph{whenever the composition can be defined}. As emphasized by Ehresmann, \cite{Ehresmann53}, each set $\D_p\subset J^p(\X\times\X)$ carries a \emph{groupoid structure}; we define the source and target of a $p$-jet $j^p\p\resx=(x,X,\ldots,X^i_L,\ldots)\in\D_p$ as
\[
\sigma(j^p\p\resx)=x\quad\text{and}\quad \tau(j^p\p\resx)=X,
\]
respectively. The groupoid multiplication of $j^p\p\resx$ and $j^p\psi\restrict{X}$, where $\tau(j^p\p\resx)=\sigma(j^p\psi\restrict{X})$, is defined as
\[
j^p(g\comp f)\resx,
\]
where $f$ and $g$ are functions in $\D$ having the $p$-jets $j^p f\resx=j^p\p\resx$ and $j^pg\restrict{X}=j^p\psi\restrict{X}$. This definition does not depend on the choice of $f$ and $g$ as can be seen from the chain rule. We write the groupoid operation as $j^p\psi\restrict{X}\cdot j^p\p\resx$. The source and target maps provide each $\D_p$ with a double fibration,

\[
\begin{tikzpicture}
  \node (A0) at (0,0) {$\D_p$};
  \node (A1) at (240:2cm) {$\X$};
  \node (A2) at (300:2cm) {$\X$.};
  \draw[->,font=\scriptsize]
  (A0) edge node[left] {$\sigma$} (A1)
  (A0) edge node[right] {$\tau$} (A2);

\end{tikzpicture}
\]
 
\begin{definition}\label{def:pseudo}
A \emph{Lie pseudo-group}, $\G$, of local transformations of $\X$ is a sub-pseudo-group of $\D$ that is determined, in each coordinate chart, by a set of formally integrable differential equations called the \emph{defining equations} that are regular in the sense of Definition \ref{regularity}. 
\end{definition}

As above, we denote the collection of transformations making up the pseudo-group by $\G$ while subscripts will indicate the set of groupoid elements, e.g. $\G_\infty$ is the set of infinite jets of transformations from $\G$. Each $\G_p$, $0\leq p \leq \infty$, is a sub-groupoid of $\D_p$. Let $\G$ be a Lie pseudo-group determined by the formally integrable equations 
\begin{equation}\label{def}
F(x,X^{(q)})=0,
\eeq
where $X^{(q)}$ denotes all jets up to order $q$. Let $\Pe$ be a one parameter family of diffeomorphisms from $\G$. The flow $\ep\mapsto \Pe(x)$ through points $x\in\X$ generates a vector field
\beq\label{vid}
\vb(x)=\zeta^i(x)\frac{\partial}{\partial x^i}
\eeq
in the \emph{Lie algebroid} $\A$ of $\G$ of local vector fields on $\X$. Note the Einstein summation convention, which will be used whenever possible. The components of $\vb$ satisfy the linearization of (\ref{def}) at the identity section $\mathbbm{1}$:
\beq\label{lin}
L(x,\zeta^{(q)})=\frac{\partial F(x,X^{(q)})}{\partial \XiK}\restrictbig{\mathbbm{1}}\zeta^i_K=0.
\eeq
As mentioned at the beginning of this section, the linearized equations are locally solvable as a result of our regularity condition. 

Now assume that the action of $\p\in\G$ on $\X$, given by $x\mapsto \p(x)$, is \emph{extended} to a trivial bundle  $\E=\X\times\U\xrightarrow{\pi}\X$ where $\U\subset \R^m$ is an open set and the action on $\U$ depends only on jets of order $1$ (it is easy to extend all results to a general integer $N>1$ but we restrict to $1$ for simplicity). This means that \beq\label{jpU}j^1\p\restrict{x}\cdot (x,u)=(\p(x), U(x,u,j^1\p\resx))\eeq where $U$ is a function of $x$, $u$ and $j^1\p\restrict{x}$. Note that we could encode this by a Lie pseudo-group of transformations, $\H$, on the bundle $\E$, of the form
\[
(x,u)\mapsto (\p(x), \psi(x,u)),
\]
where $\p$ satisfies the defining equations of $\G$ and
\[
\psi=U=U(x,u,j^1\p\resx).
\]
The Lie pseudo-group $\H$ is a \emph{one-to-one prolongation} of $\G$ in the language of \cite{stormark}. We call extended group actions of the form (\ref{jpU}) \emph{horizontal} group actions. However, working with the pseudo-group $\H$ directly will be clumsy and notationally heavy as well as unnecessarily clouding the main ideas.

By prolongation, $\G$ acts on each $J^k(\E)\xrightarrow{\pi^k}\X$, where the action on $J^k(\E)$ depends only on pseudo-group jets of order $1+k$. We refer to the jet bundle $\Jinf(\E)$ as the \emph{submanifold jet bundle}, the submanifolds in question being the sections of $\E$ which are central objects in what is to come. We denote the submanifold jet coordinates by $u^\alpha_J$, where $1\leq \alpha\leq m$ and $J\in\Nn$ and write $\zxinf=(x,u,\ldots, \uaJ,\ldots)$ for infinite jets. We then have two bundles over $\X$,

\[
\begin{tikzpicture}
  \node (A0) at (0,0) {$\X$.};
  \node (A1) at (60:2cm) {$\Jinf(\E)$};
  \node (A2) at (120:2cm) {$\G_\infty$};
  \draw[->,font=\scriptsize]
  (A1) edge node[right] {$\pi^\infty$} (A0)
  (A2) edge node[left] {$\sigma$} (A0);

\end{tikzpicture}
\]

We form the pull-back bundle $(\pi^\infty)^*\G_\infty\to\Jinf(\E)$, with induced coordinates $$(x, u, \ldots, \uaJ,\ldots,X,\ldots, X^i_K,\ldots)=(\zxinf, j^\infty\phi\restrict{x})$$ and write $\Gt_\infty$ for the pull-back $(\pi^\infty)^*\G_\infty$. We extend the source and target maps to $\Gt_\infty$ by
\[
\tilde{\sigma}(\zxinf, j^\infty\phi\restrict{x})=\zxinf,\quad \tilde{\tau}(\zxinf, j^\infty\phi\restrict{x})=\jpx\cdot\zxinf,
\] 
providing $\pbg$ with a double fibration,
\[
\begin{tikzpicture}
  \node (A0) at (0,0) {$\Gt$};
  \node (A1) at (240:2cm) {$\Jinf(\E)$};
  \node (A2) at (300:2cm) {$\Jinf(\E)$.};
  \draw[->,font=\scriptsize]
  (A0) edge node[left] {$\tilde{\sigma}$} (A1)
  (A0) edge node[right] {$\tilde{\tau}$} (A2);

\end{tikzpicture}
\]
We shall denote the target variables on $\Jinf(\E)$ by capital letters, that is
\[
Z^{(\infty)}\restrict{X}=\jpx\cdot\zxinf=\jpx\cdot(x,u,\ldots, \uaJ,\ldots)=(X,U,\ldots, \UaJ,\ldots),
\]
and when we explicitly write out the partial derivatives in $\UaJ$, they shall be capitalized also, for example
\[
j^3\p\resx\cdot(x,u,u_x,u_{xx})=(X,U, U_X, U_{XX}),
\] 
and so on.

\bre\label{pjets}
The formula for a lifted invariant depends on the jets, $X,\ldots, X^i_K,\ldots$, of pseudo-group elements $\p\in\G$. Since these jets must satisfy the defining equations of $\G$ we may replace each principal derivative by parametric ones in these formulas. After fixing a choice of principal and parametric derivatives, the latter are thought of, and referred to as the \emph{group parameters} of $\G$.
\end{Rem}

We give a few examples of horizontal actions. Most of these will be studied in detail in later sections.

\bex\label{xup}
Let $\G$ be the pseudo-group of point transformations on $J^{1}(\R\times\R\to\R)$ with coordinates $(x,u,u_x)=(x,u,p)$. This means $X$ and $U$ are functions of $(x,u)$ only and $P$ is given by
\[
P=\frac{U_x+pU_u}{X_x+pX_u}.
\]
Say we are interested in the effect of these point transformations on second order ODE $u_{xx}=q=f(x,u,p)$. Then, under $(x,u,p)\mapsto(X,U,P)$, $q$ tranforms according to 
\[
q\mapsto Q=\frac{P_x+pP_u+qP_p}{X_x+pX_u}.
\]
We are therefore in the above set-up with $\X$ parametrized by $(x,u,p)$ and $\U$ parametrized by $q$. This example generalizes, of course, to any order above two.

\eex

\bex\label{Riemannian}
In studying the local invariants of Riemannian metrics on, say, two dimensional manifolds, we are interested in the effect of a smooth change of coordinates on the components of the tensor and their jets. Let the metric in local coordinates $x\in U\subset\R^2$ be $g=g_{ij}dx^idx^j$. A change of variables is an invertible map $\p:U\to V\subset\R^2$. This will transform $g$ according to
\[
(\p^{-1})^*\left(g_{ij}dx^idx^j\right)=g_{ij}d(\p^{-1})^id(\p^{-1})^j,
\]
and so the transformation of the components $g_{ij}$ will depend on the 1-jets of $\p$. We therefore have a horizontal action where $\G$ is the pseudo-group of \emph{all} local diffeomorphisms of $\R^2$ and $\U$ is the space of metric tensors, parametrized by $g_{ij}$, $1\leq i \leq j \leq 2$, with $g_{ij}=g_{ji}$.

More generally, the components, $u$, of any tensor on a manifold will transform, under a change of variables on the base manifold, $x\mapsto\p(x)$, as a function of $x$, $u$ and $j^1\p\resx$.  
\eex

\bex
Consider the Lie pseudo-group of transformations
\[
X=f(x),\quad Y=f_x(x)y+g(x),
\]
extended to an additional real variable $u$ by
\[
U=u+\frac{f_{xx}(x)y+g_x(x)}{f_x(x)},
\]
where $f:\R\to\R$ is a local, invertible, real-analytic map and $g:\R\to\R$ is an arbitrary real-analytic map. This pseudo-group has the structure described above: The transformation of the $x$ and $y$ coordinates form a Lie pseudo-group $\G$ with defining equations $X_y=Y_y-X_x=0$ and $\G$ acts on the $u$ coordinate by $u\mapsto U= u+Y_x/X_x$. This pseudo-group is of historical interest rather than geometric as it is related to one of Medolaghi's pseudo-groups, \cite{medolaghi98}.
\eex

\bex
Let $\dis\intc L(x,u,p)dx$ be a first order Lagrangian in one variable. A point transformation $(x,u,p)\mapsto(X,U,P)$ transforms $\dis\int\!\!L(x,u,p)dx$ according to
\[
\int\!\!L(X,U,P)dX=\intc L(X,U,P)\cdot(X_x+pX_u)dx. 
\]
We can set up the problem of equivalence of Lagrangians by extending the standard pseudo-group of contact transformations on $J^1(\R^2\to\R)$ to act on a space parametrized by a real variable $L$ via
\[
L\mapsto \frac{L}{X_x+pX_u}.
\]
More generally, the divergence equivalence of Lagrangians can be cast in this framework. In this case we require that $L_{pp}$ is preserved, as opposed to $L$. See Example \ref{divel} for more.
\eex

\bex
Consider a linear second order differential operator on $\R$,
\beq\label{eq:D2}
\mathscr{D}=fD^2+gD+h,
\eeq
where $f,g,h:\R\to\R$ are real-analytic, and $f\neq0$. When we apply $\mathscr{D}$ to a real-analytic function $u:\R\to\R$ we obtain the function
\[
fu''+gu'+hu.
\]
Now consider the pseudo-group, $\G$, of transformations of the $(x,u)$ of the form
\[
(x,u)\mapsto (\p(x), u\cdot\psi(x))=(X, U),
\]
where $\p, \psi:\R\to\R$ are real-analytic. Restricting to the set $\X=\{(x,u)\in\R^2~|~u>0\}$ and to $\psi>0$ this is a Lie pseudo-group of transformations of $\X$ that has defining equations
\[
X_y=0,\quad U_{uu}=0,\quad U_{u}u=U.
\]
The elements of $\G$ preserve the space of linear operators and a transformation from $\G$ maps (\ref{eq:D2}) to
\[
\bar{\mathscr{D}}=F\bar{D}^2+G\bar{D}+H,
\]
where $\bar{D}$ is the derivative with respect to the transformed independent variables $X$ and the lifted coefficients $F, G, H$ have explicit formulas
\beq\label{eq:FGH2}
\begin{aligned}
F&=f\frac{X_x^2}{U_u},\\
G&=-f\frac{2 U_xX_x-X_{xx}U_uu}{uU_u^2}+g\frac{X_x}{U_u},\\
H&=-f\frac{U_{xx}U_uu-2U_x^2}{u^2 U_u^3}-g\frac{U_x}{u U_u^2}+\frac{h}{U_u}.
\end{aligned}
\eeq
Notice that since $\dis\frac{U_x}{u}$ and $U_u$ are independent of $u$, these lifted invariants are also independent of $u$ (as they should be be). Each operator $\mathscr{D}$ defines a section,
\[
(x,u)\mapsto (x,u, f(x), g(x), h(x)),
\]
in the trivial bundle $\X\times\R^3\to\X$ and two operators are equivalent if their respective sections are congruent under the extended action of $\G$ given by (\ref{eq:FGH2}).
\eex

\end{subsection}

\begin{subsection}{Lifted invariants}\label{sec:lift}

A local transformation $\psi\in\G$ acts by right groupoid multiplication on the set of jets $\jpx\in\G_\infty$ with $x\in\text{dom}~\psi$ by 
\[
R_{\psi}\cdot \jpx=j^\infty\p\resx\cdot j^\infty\psi^{-1}\restrict{\psi(x)},
\]
and from the left on the set of jets $\jpx\in\G_\infty$ with $\tau(j^\infty\p\resx)\in\text{dom}~\psi$ by
\[
L_{\psi}\cdot\jpx=j^\infty\psi\restrict{\tau(j^\infty\p\resx)}\cdot j^\infty\p\restrict{x}.
\]
Note that projecting the right action onto the source coordinate gives $$ \sigma(R_{\psi}\cdot \jpx)=\psi(x)$$ and so we can extend this action from $\G_\infty$ to $\pbg$ by 
\beq\label{eq:Rpsi}
R_\psi\cdot (\zxinf, \jpx):=(j^\infty\psi\restrict{x}\cdot\zxinf, j^\infty\p\resx\cdot j^\infty \psi^{-1}\restrict{\psi(x)}),
\eeq
for all $x$ in the domain of definition of $\psi$. The target map $\tilde{\tau}:\pbg\to J^\infty(\E)$ provides a complete collection of all scalar invariants of this action:
\begin{align*}
R_\psi^*\tilde{\tau}(\zxinf,\jpx)&=\tilde{\tau}(j^\infty\psi\restrict{x}\cdot \zinf\resx, j^\infty\p\resx\cdot j^\infty\psi^{-1}\restrict{\psi(x)})\\
&= j^\infty\p\resx\cdot j^\infty\psi^{-1}\restrict{\psi(x)}\cdot (j^\infty\psi\restrict{x}\cdot \zinf\resx)\\
&= \left(j^\infty\p\resx\cdot j^\infty\psi^{-1}\restrict{\psi(x)}\cdot j^\infty\psi\restrict{x}\right)\cdot \zxinf=\jpx\cdot\zxinf\\
&=\tilde{\tau}(\zxinf, \jpx).
\end{align*}

This also means that the pull-back of any differential form $\o$ on $\Jinfe$ by the target map is invariant under this action:
\beq\label{Rtau}
R_\psi^*(\tilde{\tau}^*\o)=(\tilde{\tau}\comp R_\psi)^*\o=\tilde{\tau}^*\o.
\eeq

Turning our attention back to the diffeomorphism groupoid $\D_\infty$, in \cite{olver05} a basis, $\{\muiK\}_{K\in\Nn, 1\leq i\leq n}$, for the contact co-distribution on $\D_\infty$, that is \emph{right invariant} (under the action of $\D$) was constructed. Naturally, these contact forms are called the Maurer-Cartan forms of the pseudo-group $\D$. The form $\muiK$ agrees with the standard contact form $\upiK=dX^i_K-X^i_{K,j}dx^j$ on the identity section $\mathbbm{1}$ of $\D_\infty\to\X$, and each $\muiK$ is a linear combination of $\upjJ$ for $|J|\leq |K|$ (and conversely). When we restrict the Maurer-Cartan forms to a sub-groupoid $\G_\infty\subset\D_\infty$, we obtain certain linear dependencies among the $\muiK$. The important discovery, made in \cite{olver05}, is that these are given by
\beq\label{linmc}
\frac{\partial F(x,X^{(q)})}{\partial \XiK}\restrictbig{\widetilde{\mathbbm{1}}}\muiK=0,
\eeq
where restriction to $\widetilde{\mathbbm{1}}$ means first restricting to the identity section and then replacing all source coordinates $x$ by target coordinates $X$.

\bre\label{strrem}
The structure equations, or the formulas for $d\muiK$, are rather complicated expressions but we will mostly be interested in their top order terms. The entire  equations are 
\beq\label{structureeq}
d\muiK=\sum_{1\leq j\leq n}\o^j\wedge\muiKj+\sum_{\substack{L+M=K \\ |M|\geq1}}\binom{K}{L}\sum_{1\leq j\leq n}\mu^i_{L,j}\wedge\mu^j_M.
\eeq
Where we refer to $\sum_{1\leq j\leq n}\o^j\wedge\mu^i_{K,j}$ as the top order term, $L+M$ is the componentwise addition of multi-indices in $\Nz^n$ and
\[
\binom{K}{L}=\frac{K!}{L!M!}.
\] 
\end{Rem}

The Maurer-Cartan forms embed naturally into the pull-back bundle $\Gt_\infty\to\Jinf(\E)$. In fact, the space of one-forms on $\Gt_\infty$ is a direct sum of two \emph{right-invariant} (recall (\ref{eq:Rpsi})) subspaces. Complementing the Maurer-Cartan forms in this direct sum are the one-forms on $\Jinf(\E)$, pulled back to $\Gt_\infty$ by the target map $\taut$. This direct sum imbues the differential forms on $\Gt_\infty$ with a bigration. We correspondingly \emph{split} the exterior derivative on $\Gt_\infty$ into a \emph{group}- and \emph{jet}-component,
\[
d=d_G+d_J,
\]
where $d_G$ increases the group-grade and $d_J$ increases the jet-grade. Notice that since $\psi\in\G$ acts on the group and submanifold jet coordinates separately, $R_\psi\cdot(\zxinf, \jpx)=(j^\infty\psi\restrict{x}\cdot\zxinf, j^\infty\p\resx\cdot j^\infty\psi^{-1}\restrict{\psi(x)})$, the group and jet components of the derivative of any invariant form are again invariant. The differential one-forms on $\Jinf(\E)$ further divide into horizontal and contact forms with bases
\[
\{dx^1,\ldots,dx^n\}\quad\text{and}\quad\{d\uaJ-\uaJi dx^i\}_{J\in\Nn,1\leq \alpha\leq m},
\] 
respectively. The jet-differential $d_J$ then splits accordingly into a horizontal and vertical (submanifold contact) component,
\[
d_J=d_H+d_V.
\]
Since the right action of $\G$ on $\Gt_\infty$ obviously preserves horizontal and vertical forms (since $\G$ acts by contact transformations on $\Jinf(\E)$), as well as the Maurer-Cartan forms, these various differentials of invariant forms on $\Gt_\infty$ are still invariant. 

Let us define an operator, $\gamma_J$, that takes a general differential form on the pull-back bundle $\Gt_\infty\to\Jinf(\E)$, that is written in terms of the Maurer-Cartan forms, horizontal forms and submanifold jet contact forms, and equates all Maurer-Cartan forms to zero. Recall that $\tilde{\tau}^*\o$ is right invariant. Since the right action preserves the group and submanifold jet components of $\taut^*\o$ we have that $\gamma_J\taut^*\o$ is invariant. We call the operator
\beq\label{lift}
\lambda:=\gamma_J\taut^*
\eeq  
\emph{the lift operator}. We say that a differential form $\Omega$ on $\pbg$ is \emph{concentrated} on $\Jinf(\E)$ if $\gamma_J(\Omega)=\Omega$. Notice that $\lambda$ maps differential forms on $\Jinf(\E)$ to invariant differential forms on $\Gt_\infty$ that are concentrated on $\Jinf(\E)$. 

\end{subsection}

\begin{subsection}{Recurrence formula}
In the calculus of moving frames for Lie pseudo-groups, \emph{the recurrence formula} plays a fundamental role. We now deduce it for horizontal actions, but the proofs for general Lie pseudo-group actions are identical.

Let $\Pe=(\Pe^1,\ldots, \Pe^n)$ be a one parameter family of local diffeomorphisms in $\G$ with $\Phi_0$ being the identity. It generates a local vector field on $\X$,
\[
\vb(x)=\frac{d}{d\ep}\restrictbig{\ep=0}\Pe^i(x)\frac{\partial}{\partial x^i}=\zeta^i(x)\frac{\partial}{\partial x^i},
\] 
whose components, $\zeta^i$, satisfy (\ref{lin}). The domain of definition of $\vb$ is the set
\[
\text{dom}~\vb=\bigcup_{t>0}\bigcap_{|\ep|<t}\text{dom}~\Pe.
\] 
By prolongation, we have the flow $\ep\mapsto \taut(\zxinf, j^\infty\Pe\restrict{x})$ on $J^\infty(\E)$, through points $\zxinf\in(\pi^\infty)^{-1}(\text{dom}~\vb)$, generating the local vector field $\widehat{\vb}^\infty(\zxinf)$. The explicit formulas for the different components of $\widehat{\vb}^\infty$ can be given by a simple recurrence relation, see \cite{olver00}.

The \emph{lift} of the flow $\Pe$ is the flow on $\G_\infty$ given by $\ep\mapsto j^\infty\restrict{X}\Pe\cdot j^\infty\psi\restrict{x}$ for $j^\infty\psi\resx\in \G_\infty$ with $\tau(j^\infty\psi\resx)=X$ in dom $\Pe$. Note that it is tangent to the source fibers $\sigma^{-1}(x)$. This flow is obviously right-invariant and analyzing it at the \emph{identity section}, $\one\subset\G$ we find: $\displaystyle \widehat{\mathbf{V}}^\infty\restrict{\one}=\sum\zeta^i_K(X)\frac{\partial}{\partial \XiK}$. Note that, by definition, $\Vbinf$ and $\vbinf$ are related by the push-forward at the identity, $\taut_*\left(\Vbinf\restrict{\one}\right)=\vbinf$, where, implicitly, we have transferred $\Vbinf$ from $\G_\infty$ to $\Gt_\infty$. We shall also denote by $\one$ the pull-back of the bundle $\one\to\X$ to $\one\to\Jinf(\E)$, a subbundle of $\Gt_\infty$, and refer to it as \emph{the identity section} of $\Gt_\infty$.

We would like to know what happens when we take the exterior derivative of a lifted form $\lambda(\o)$, where $\o$ is a differential form on $\Jinf(\E)$. It is relatively easy (by an examination of $\lambda$) to see that $d_J\lambda(\o)=\lambda(d\o)$. The group component $d_G\lambda$ is more difficult to establish, but we have, for a lifted Lie algebroid vector field $\Vbinf$, since $\taut_*\left(\Vbinf\restrict{\one}\right)=\vbinf$, and since $\Vbinf$ has only group components,  at the identity section, 
\beq\label{vb}
\Vbinf(\lambda(\o))=\Vbinf(\gamma_J\taut^*\o)=\gamma_J\Vbinf(\taut^*\o)=\gamma_J\taut^*(\vbinf(\o)).
\eeq
 On the other hand, we have by Cartan's formula
\[
\Vbinf(\lambda(\o))=\Vbinf\interior d\lambda(\o) + d(\Vbinf\interior \lambda(\o)) = \Vbinf\interior d_G\lambda(\o),
\]
again, since $\Vbinf$ only has group components (i.e. is tangent to the $\tilde{\sigma}$-fibers). Now, $\Vbinf\interior d_G\lambda(\o)=\Vbinf\interior\sum(\mu^i_K\wedge\Omega^K_i)$ is an invariant differential form on $\Gt_\infty$, where $\Omega^K_i$ are some differential forms concentrated on $\Jinfe$. Evaluating it at the identity section gives
\beq\label{Vb}
\Vbinf(\lambda(\o))\resid=\sum\zeta^i_K(X)\Omega^K_i.
\eeq
Both (\ref{vb}) and (\ref{Vb}) are invariant differential forms that agree at the identity section, but this means they must agree everywhere. Writing $\gamma_J\taut^*(\vbinf(\o))=\sum\zeta^i_K(X)\widetilde{\Omega}^K_i$ (note that $\vbinf$ is a linear function of the components $\zeta^i_K$), we have, for every vector field in the Lie algebroid of $\G$ that
\[
\sum\zeta^i_K(X)\widetilde{\Omega}^K_i=\sum\zeta^i_K(X)\Omega^K_i,
\]
and hence $\widetilde{\Omega}^K_i=\Omega^K_i$. This gives the original recurrence formula, where on the left hand side we need to replace every $\lambda(\zeta^i_K)$ by $\mu^i_K$:
\beq\label{vbVb}
\lambda(\vbinf(\o))=d_G\lambda(\o).
\eeq
This is the original derivation of the recurrence formula for pseudo-groups from \cite{olver08moving}. But we can also calculate directly,
\[
\Vbinf(\lambda(\o))\restrictbig{\mathbbm{1}}=\sum\zeta^i_K(X)\left(\frac{\partial (\lambda(\o))}{\partial X^i_K}\right)\restrictbig{\mathbbm{1}},
\]
to see that $\displaystyle \Omega^K_i\restrictbig{\mathbbm{1}}=\left(\frac{\partial (\lambda(\o))}{\partial X^i_K}\right)\restrictbig{\mathbbm{1}}$. Each $\Omega^K_i$ is invariant on $\Gt_\infty$ and since right invariant differential forms on $\Gt_\infty$ are determined by their values on the identity section of $\Gt_\infty$ we can deduce that, in general,
\[
\Vbinf(\lambda(\o))=\sum\zeta^i_K(X)\lambda\left(\frac{\partial (\lambda(\o))}{\partial X^i_K}\restrictbig{\mathbbm{1}}\right).
\]
We can then write the recurrence formula
\beq\label{rec}
d_G\lambda(\o)
=
\lambda(\vbinf(\o))
=
\sum\mu^i_K\wedge\lambda\left(\frac{\partial (\lambda(\o))}{\partial X^i_K}\restrictbig{\one}\right).
\eeq
This form of the recurrence formula displays clearly the symbol structure of $\G$ and can be used to prove termination of the forthcoming equivalence method, cf. \cite{IMF}. However, since $\vbinf$ can be computed using the \emph{prolongation formula}, \cite{olver00}, we can easily obtain the algebraic structure of our partial moving frames, cf. Section \ref{examples1}. 

Notice that (\ref{rec}) is a kind of group parameter linearization of $\lambda(\o)$, e.g. if $\o=\uaJ$ we have
\beq\label{u}
d_G\lambda\uaJ=d_G\UaJ=\sum\lambda\left(\frac{\partial \UaJ}{\partial X^i_K}\restrictbig{\one}\right)\mu^i_K.
\eeq

The invariant horizontal forms $\oi:=\lambda(dx^i)$ are especially important. We have
\[
\oi=\lambda(dx^i)=\gamma_J\taut^*dx^i=\gamma_JdX^i=\gamma_J\left(X^i_jdx^j-\Upsilon^i\right)=X^i_jdx^j.
\]
Another important identity is
\beq\label{dhu}
d_H\UaJ=\UaJi\oi,
\eeq
which can be deduced by computing
\begin{align*}
{}&d_H\lambda(\uaJ)+d_V\lambda(\UaJ)=d_J\lambda(\uaJ)=\lambda(d\uaJ)\\
&=\lambda(\uaJi dx^i)+\lambda(\text{vertical forms})\\
&=\UaJi\oi+\text{vertical forms}\quad~~~(\lambda~\text{preserves vertical forms}),
\end{align*}
and comparing horizontal parts.

Combining (\ref{u}) and (\ref{dhu}) we have
\beq\label{du}
d\UaJ=\UaJi\oi+\sum\lambda\left(\frac{\partial \UaJ}{\partial X^i_K}\restrictbig{\one}\right)\mu^i_K+\text{contact forms on}~\Jinf(\E).
\eeq

\end{subsection}
\end{section}

%% file: imf-Core.tex
\begin{section}{Equivalence of sections}\label{core}

 In this section we shall first introduce the equivariant moving frame for pseudo-groups, which, in conjunction with the recurrence formula will be our fundamental tool. Section \ref{SGp} gives a proof of a pseudo-group analog of the result for Lie groups that underlies the congruence problem in the finite dimensional case and the following section demonstrates how Cartan's $G$-structures naturally emerge from that vantage point.

%


\begin{subsection}{The equivariant moving frame}\label{sec:equi}
The lifted submanifold jet coordinate functions $\UaJ=\tau^*\uaJ$ are invariant under the right action of $\G$ on the bundle $\pbg\to\Jinf(\E)$, and form a complete collection of invariants for this action. The equivariant moving frame, \cite{olver08moving}, is an object that computes the invariants of the prolonged action of $\G$ on the \emph{submanifold jet bundles} $J^r(\E\overset{\pi^r}{\to}\X)$. Unfortunately, a moving frame can usually not be defined on all of $J^r(\E)$ and one must partition $J^r(\E)$ into disjoint sets, each of which carrying its own moving frame. So let the groupoid $\G_p$ act on $J^r(\E)$, for some $1\leq r,p\leq\infty$. For $\S\subset J^r(\E)$ we denote by $\G^\S_p$ the jets $j^p\p\resx$, $x\in\pi^r\left(\S\right)$, that preserve $\S$ under the right action. Notice that $\G^\S_p$ is a groupoid. Denote the collection of local diffeomorphisms $\psi\in \G$ that preserve $\S\in J^r(\E)$ by $\G^\S$. We also form the pull-back bundles
\[
\Gt^\S_p\overset{\tilde{\sigma}^p_r}{\longrightarrow}\S,
\]
whose fiber over $j^ru\resx\in\S$ is the set of all $j^p\p\resx\in\G^\S_p$ with source coordinate $x$ and on which $\G^\S$ acts by
\[
R_\psi\cdot(j^p\p\resx, j^ru\resx)=(j^p\p\resx\cdot j^p\psi^{-1}\restrict{\psi(x)}, j^p\psi\resx\cdot j^ru\resx).
\] 
We give the following, more general, definition of a moving frame than the original, \cite{olver05}. This definition will allow us to solve equivalence problems for singular jets, which are not within reach of the original method of the equivariant moving frame (cf. Remark \ref{singular_ex}). 

\begin{definition}
Let $\G_p$ act on $J^r(\E)$. A (local) $p\th$ order moving frame with \emph{domain of definition} $\S\subset J^r(\E)$ is a (local) $\G^\S$-right-equivariant section, $\rho$, of $\Gt^\S_p\overset{\tilde{\sigma}^p_r}{\longrightarrow}\S$, i.e.
\beq\label{rho}
\rho(j^p\psi\restrict{x}\cdot j^ru\resx)=R_\psi\cdot\rho(j^ru\resx),
\eeq
for all $\psi\in\G^\S$ and $j^ru\resx\in\S$.
\end{definition} 
Letting $\psi\in\G^\S$, and abusing language somewhat, we can write (\ref{rho}) as $$\rho\comp\psi=R_\psi\comp\rho$$ and so right-equivariance is the property of a section of $\Gt^\S_p\to\S$ that it commutes with the action of $\G^\S$. Or, in other words, its image is invariant under the right action of $\G^\S$ on $\Gt^\S_p$.

\begin{Rem}
Notice that there is no guarantee that $\G^\S_p$ is the space of $p$-jets of elements of $\G^\S$ as we have not stipulated any local solvability of the differential equations determining the pseudo-group $\G^\S$. Indeed, the goal of Cartan's equivalence method, and, by extension, our forthcoming method, is deducing an involutive system for $\G^\S$. See the examples in Section \ref{examples1}.
\end{Rem}

A moving frame pulls invariant differential forms (under the right action action of $\G^\S$) on the pull-back bundle $\Gt^\S_p\to\S$ back to an invariant object on $\S$ as can be seen as follows. Let $\Omega$ be an invariant differential form on $\Gt^\S_p$ and let $\o=\rho^*\Omega$ be its pull-back on $\S$. Then, for $\psi\in\G^\S$,
\[
\psi^*\o=\psi^*\rho^*\Omega=(\rho\comp\psi)^*\Omega=(R_\psi\comp\rho)^*\Omega=\rho^* R_\psi^*\Omega=\rho^*\Omega=\o,
\]
by invariance of $\Omega$. In particular, $\rho^*\UaJ$, $|J|\leq r$, are a complete collection of invariants of the action of $\G^\S$ on $\S$. Generalizing a little bit, we define a \emph{partial} moving frame.

\begin{definition}\label{parmov}
A (local) $p\th$ order partial moving frame on $\S\subset J^r(\E)$ is a fibered subspace, $\B\overset{\tilde{\sigma}^p_r}{\longrightarrow}\S$ of $\Gt^\S_p\to J^r(\E)$, that is preserved by the right the action of $\G^\S$. The set $\S$ is called the domain of definition of $\B$ and we say that $\B$ is \emph{right-equivariant} under the action of $\G^\S$.
\end{definition}

Notice that if $\G_p$ acts on $J^r(\E)$ then $\G_t$ also acts on $J^r(\E)$ for all $p\leq t\leq \infty$. We can then form the pull-back bundles $\Gt_t\to \S$ on which $\G^\S$ (and $\G^\S_t$) acts. There are natural bundle maps $\Gt_t\overset{\nu^t_p}{\to}\Gt_p$, $p\leq t\leq\infty$, defined by
\[
(j^t\p\resx, j^ru\resx)\overset{\nu^t_p}{\longmapsto} (j^p\p\resx, j^ru\resx).
\]

\begin{definition}\label{Bpreimage}
Let $\B_p\to\S$ be a $p\th$ order partial moving frame. Denoting its preimage in $\Gt_\infty$ by $\Bt_p:=\left(\nu^\infty_p\right)^{-1}\left(\B_p\right)$ we notice that $\Bt_p$ is a partial moving frame in $\Gt_\infty\to\S$.
\end{definition}

\bre\label{rem:GS}
Definition \ref{parmov} is more general than previous definitions of partial moving frames which are usually only defined as fibered subspaces over a base $\S$ where $\S$ is assumed to be locally $\G$-invariant. Our definition is strictly more general as (locally) $\G^\S=\G$ for locally $\G$-invariant $\S$. Furthermore, our key result in the next subsection allows us to extend the moving frame technique to equivalence problems for sections that have until now remained outside the scope of the equivariant moving frame. 
\end{Rem}

Just like a moving frame pulls invariant objects on $\Gt_p$ to invariant objects on $\S$, restricting (pulling-back) any invariant object on $\Gt_p$ to a partial moving frame $\B_p\to\S$ gives an invariant object on $\B_p$ under $\G^\S$.

There turns out to be a practical construction available for a partial moving frame that we, for simplicity, demonstrate for $\G_\infty$ acting on $\Jinf(\E)$. The \emph{orbit} of the action of $\G$ on $\Jinf(\E)$ through $\jinf u\resx$ is the set $$\{\jinf\p\resx\cdot\jinf u\resx~|~\jinf\p\resx\in\G_\infty\}.$$ The construction of a moving frame is equivalent to a choice of a \emph{cross-section} to the orbits of the action of $\Jinf(\E)$. A cross-section $\K\subset\Jinf(\E)$ to the orbits is a (connected) subspace such that if an orbit intersects $\K$, it does so at a unique point and transversally. Given such a cross-section we can construct a moving frame as follows. Let $\S\subset\Jinf(\E)$ be the set of $\jinf u\resx$ whose orbits intersect $\K$ (this set is seldom all of $\Jinf(\E)$ and so each cross-section determines a different set $\S$). For $\jinf u\resx\in\S$ define the fiber (in the partial moving frame) over $\jinf u\resx$ to be the collection of $\jpx\in\G_\infty$ such that 
\[
\jpx\cdot\jinf u\resx\in\K.
\] 
Note that each cross-section defines a set $\S$. If the action is free on $\S$, then this partial moving frame reduces to a proper moving frame since then $\jpx\cdot\jinf u\resx\in\K$ \emph{uniquely} determines $\jpx$. The resulting subspace, $\B\to\S$, is right-equivariant. To see this note that if $(\jinf u\resx, \jpx)\in\B$ then, for $\psi\in\G^\S$,
\[
R_\psi\cdot(\jinf u\resx,\jpx)=(j^\infty\psi\restrict{x}\cdot\jinf u\resx, j^\infty\p\resx\cdot j^\infty\psi^{-1}\restrict{\psi(x)}),
\]
and 
\[
j^\infty\p\resx\cdot\psi^{-1}\restrict{\psi(x)}\cdot\left(j^\infty\psi\restrict{x}\cdot\jinf u\resx\right)=\jpx\cdot\jinf u\resx\in\K,
\]
so $R_\psi\cdot(\jinf u\resx, \jpx)\in\B$, and $\B$ is right-equivariant. 

\begin{Rem}\label{rem:orbit_regularity}
For there to exist a (local) cross-section to the pseudo-group orbits, the Lie pseudo-group must satisfy certain non-trivial criteria. For example, an irrational flow on a torus has no local cross-section since the intersection of an orbit and any open neighborhood in the torus consists of infinitely many connected components. For our constructions we must assume that local cross-sections to our orbits exist, and in applications one must check this on a case by case basis.
\end{Rem}

In practice the cross-section $\K$ is usually built, order-by-order, as the subspace where an increasing number of the jet-coordinates on $\Jinf(\E)$, $\uaJ$, are constant. This will give a decreasing sequence of partial moving frames
\beq\label{lest}
\Gt_\infty\supset\Bt_0\supset\Bt_1\supset\Bt_2\supset\ldots
\eeq
The partial moving frames are then described by the solutions to equations of the form
\[
\UaJ=\text{constant},
\]
and we say we have \emph{normalized} the lifted invariant $\UaJ$ when $\uaJ$ is constant on $\K$. 

\begin{Rem}\label{powder}
Much effort was devoted in \cite{olver09} to deciding when a sequence of partial moving frames as in (\ref{lest}) \emph{converges} to a bona-fide moving frame. That paper's key result was dubbed \emph{persistence of freeness of Lie pseudo-groups} and was achieved through a rather difficult analysis of convoluted algebraic objects such as ``eventual polynomial modules''. From our ``involutive'' point of view we have no need for the results of \cite{olver09} and we achieve their generalization by straightforwardly applying the Cartan-Kuranishi completion theorem, in the guise of Algorithm \ref{galgo}.  
\end{Rem}

\bre\label{groupparameter}
When restricting a partial moving frame $\Bt_p$ to a new equation $\UaJ=c$, we \emph{solve} $\UaJ=c$ for one of the group parameters. This group parameter then disappears from our parametrization of $\Bt_p$, and, by a slight risk of confusion, we say that this group parameter has been normalized. 
\end{Rem}

\bex
Consider the Lie pseudo-group action of horizontal transformations on sections of $\R^3\to\R^2$ obtained by extending the Lie pseudo-group of transformations on $\R^2$, $\G$, with elements
\[
X=f(x),\quad Y=f_x(x)y+g(x),
\]
to act on a variable $u\in\R$ such that
\[
U=u+\frac{Y_x}{X_x}.
\]
The pseudo-group $\G$ has determining equations
\[
X_y=0,\quad Y_y=X_x.
\]
We build the cross-section order-by-order, first setting 
\[
\K=\{x=y=u=0\}.
\]
A point $(\zxinf, \jpx)=(x,y,u,\ldots, \uaJ, \ldots, X,Y, \ldots,X_K, Y_K,\ldots)$ is in the corresponding partial moving frame, $\Bt_1\subset\Gt_\infty$, if and only if
\begin{align*}
{}&X=0,\quad Y=0,\quad U=0\\
\iff~&X=0,\quad Y=0, \quad Y_y=-uX_x.
\end{align*}
At the next order, we have lifted invariants
\begin{align*}
{}&U_X=\frac{u_x}{X_x}+\frac{Y_{xx}X_x-X_{xx}Y_x}{X_x^3},\\
&U_Y=-\frac{Y_x}{X_x}U_X+\frac{u_y}{X_x}+\frac{Y_{xy}}{X_x^2}.
\end{align*}
Notice that we are using the fact that the pseudo-group jets are coming from a pseudo-group with defining equations
\[
X_y=0,\quad Y_x=X_x,
\]
and we have replaced all principal derivatives by parametric ones in our formulas for lifted invariants. We can normalize both $U_X=U_Y=0$ to obtain a partial moving frame $\Bt_2\subset\Bt_1\subset\Gt$ on which
\[
Y_{xx}=\frac{X_{xx}Y_x-X_x^2u_x}{X_x},\quad Y_{yx}=-u_yX_x.
\]
We can continue like this, normalizing $U_{XX}=U_{XY}=0$, $U_{YY}=1$ and $U_{XXX}=U_{XXY}=0$ (skipping the details of the calculations; it is a little bit of work). At this point we have actually normalized \emph{all} pseudo-group parameters of order at most 4. 

Restricting the invariants $U_{XYY}$ and $U_{YYY}$, on $\Gt$, to the partial moving $\Bt_4$ gives the genuine invariants
\beq\label{uyy2}
U_{XXX}\mapsto\frac{u_{xyy}+uu_{yyy}+2u_yu_{yy}}{u_{yy}^{3/2}},\quad U_{YYY}\mapsto \frac{u_{yyy}}{u_{yy}^{3/2}}.
\eeq
This gives the rough idea of how the method of equivariant moving frames proceeds. We will redo this as Example \ref{medo} and shall see how (our modification of) Cartan's equivalence method drastically decreases the computational load of the above routine.   
\eex

\bre\label{cupS}
Notice that in (\ref{uyy2}) the jet coordinate $u_{yy}$ must be positive, and so, at some point of our normalization process, we made the decision to restrict our partial moving frame to the subspace of $\Jinf(\E)$ where $u_{yy}>0$. For jets with $u_{yy}<0$ some of the normalizations made in the above example were not possible, and we could not have constructed this particular partial moving frame. In general, the space $\Jinf(\E)$ must be partitioned into a collection of subsets 
\[
\Jinf(\E)=\bigcup_{i=1}^N\S_i,
\] 
where we obtain a different partial moving frame on each of $\S_i$ which corresponds to a different cross-section to the pseudo-group orbits. In the above example, we have the partition $\S_1=\{\zinf\resx~|~u_{yy}> 0\}$, $\S_2=\{\zinf\resx~|~u_{yy}= 0\}$ and $\S_3=\{\zinf\resx~|~u_{yy}< 0\}$. The subset $\S_2$ is ``singular'' in the sense that $\S_2$ is not a locally $\G$-invariant set. Our combination of Cartan's equivalence method and the equivariant moving frame, built on the results in the next subsection, allows for analysis of these singular jets, cf. Section \ref{examples1}.
\end{Rem}

\end{subsection}

\begin{subsection}{Sections of $\G_p$}\label{SGp}
A well-known fact is that two submanifolds, of the same dimension, in a Lie group are congruent if and only it there exists a map between them that preserves the pulled-back Maurer-Cartan forms, \cite{griffiths74}. This result underlies the method of moving frames for Lie groups, \cite{Fels99}. In this subsection we shall prove the infinite dimensional analog of this fact, which will underlie our eventual equivalence method.

Consider a Lie pseudo-group, $\G$, of local transformations on the manifold $\X$, determined by formally integrable and regular and hence locally solvable differential equations
\[
F(x,X\qv)=0.
\]
As noted above, at each order $\G_{q+t}:=\G_{q,t}$ carries a groupoid structure, the groupoid elements of $\G_p$ being the $p$-jets $j^p\p\resx$ for $\p\in\G$. (In the following, it will sometimes be convenient to denote these groupoid elements by lower case Latin letters such as $g$ and $h$.) As mentioned before, the source and target maps endow $\G_p$ with a double fibration
\[
\begin{tikzpicture}
  \node (A0) at (0,0) {$\G_p$};
  \node (A1) at (240:2cm) {$\X$};
  \node (A2) at (300:2cm) {$\X$.};
  \draw[->,font=\scriptsize]
  (A0) edge node[left] {$\sigma$} (A1)
  (A0) edge node[right] {$\tau$} (A2);

\end{tikzpicture}
\]
We shall denote the source and target fibers by
\[
\sigma^{-1}(x)=\G_p\resx\quad\text{and}\quad \tau^{-1}(X)=\G_p\irestrict{X}.
\]

Now consider two local sections of $\G_p\to\X$, $s$ and $\sb$. (Note that $s$ and $\sb$ are \emph{not necessarily} (indeed, in practice, never will be) the graphs of prolongations of transformations in $\G$, i.e. contact forms do not vanish when restricted to their images.) We want to know whether there exists a local transformation $\p\in\G$ such that 
\[
R_\p\cdot s=\sb.
\]

Immediate invariants for this problem are the target coordinates of $s$ and $\sb$, since the right-action leaves these invariant. For the time being we consider only sections $s$ and $\sb$ that have constant, and equal, target coordinates,
\beq\label{tausx}
\tau(s(x))=\tau(\sb(\xb))=X_0=\text{constant},
\eeq
but our results will trivially extend to sections with arbitrary target
coordinates (cf. Theorem \ref{Gco_gen} below). Notice that $s$ and $\sb$
satisfying (\ref{tausx}) are sections of the bundle
$\G_p\irestrict{X_0}\overset{\sigma}{\to} \X$.
Obviously, the tangent vectors to the images of $s$ and $\sb$ will then have
zero target component. Being tangent to the target fibers, we shall call such
vectors in $T\G_p$ \emph{$\tau$-vertical}. Tangent vectors that have zero
source-components shall be called vertical. Notice that
$T\left(\G_p\resx\right)$ is the space of vertical tangent vectors at the source
coordinate $x$, while $T(G_p\irestrict{X})$ is the space of
$\tau$-vertical vectors at the target coordinate $X$.

As we have seen, a local transformation $\p\in\G$ acts on $\G_p$ by the left and right actions:
\beq\label{RLp}
\begin{aligned}
{}&R_{\p}\cdot j^p\psi\restrict{x}=j^p\psi\restrict{x}\cdot j^p\p^{-1}\restrict{\p(x)},\quad\text{for all $j^p\psi\restrict{x}$ with source in the domain of $\p$},\\
&L_{\p}\cdot j^p\psi\restrict{x}=j^p\p\restrict{\tau(j^p\psi\resx)}\cdot j^p\psi\restrict{x},\quad\text{for all $j^p\psi\restrict{x}$ with target in the domain of $\p$}.
\end{aligned}
\eeq
On the other hand, for a single groupoid element $g\in\G_p$ with $\sigma(g)=x$ we can define the map $R_g\cdot h=h\cdot g^{-1}$ for all $h$ with $\sigma(h)=x$ (where $\cdot$ is the groupoid multiplication), i.e. from the source fiber $\G_p\restrict{\sigma(g)=x}$ to $\G_p\restrict{\tau(g)}$. This map is real-analytic (indeed, it is algebraic) and its derivative is a map between vertical vectors:
\beq\label{Rg}
R_{g_*}:T_g\G_p\restrict{x}\to T_{h\cdot g^{-1}}G_p\restrict{\tau(g)}.
\eeq
Similarly, we can define the map $L_g\cdot h=g\cdot h$ for all $h\in\G_p$ with $\sigma(g)=\tau(h)$. Its differential is a map between $\tau$-vertical vectors:
\beq\label{Lg}
L_{g_*}:T_g\G_p\irestrict{x}\to T_{g\cdot h}G_p\irestrict{\tau(g)}.
\eeq

A rather trivial, but important, observation is that the derivative of the right and left actions of a local transformation $\p\in\G$ agree with the derivatives of the groupoid actions (\ref{Rg}) and (\ref{Lg}) when restricted to vertical, and $\tau$-vertical vectors, respectively.

\begin{lemma}\label{same}
Let $g\in\G_p$ be a groupoid element with source $x$ and target $X$. Given a vertical tangent vector $V\in T_g\G_p\resx$ and a local transformation $\p\in\G$ with domain including $x$, we have
\[
R_{\p_*}V=R_{(j^p\p\resx)_*}V,
\]
Where the left hand side is the derivative of the map $R_\p$, and the right hand side by (\ref{Rg}). Similarly, for a $\tau$-vertical vector $V\in T_g\G_p\irestrict{X}$ we have
\[
L_{\p_*}V=L_{(j^p\p\resx)_*}V.
\]  
\end{lemma}

\begin{proof}
Let $\Phi(\ep)$ be a path in $\G_p$ such that $\frac{d}{d\ep}\restrictbig{\ep=0}\Phi(\ep)=V$ and such that $\sigma(\Phi(\ep))=x$ for all $\ep$. Then we have
\beq\label{ep1}
R_\p\cdot \Phi(\ep)=\Phi(\ep)\cdot j^p\p^{-1}\restrict{\p(x)},
\eeq
but since each $\Phi(\ep)$ has source $x$, this is trivially equal to
\beq\label{ep2}
R_{j^p\p\resx}\cdot \Phi(\ep),
\eeq
and the result follows by differentiating (\ref{ep1}) and (\ref{ep2}) with respect to $\ep$ and evaluating at $\ep=0$. The second part proceeds similarly. 
\end{proof}

Turning to the equivalence problem of sections with constant target coordinates, let $s$ and $\sb$ be two sections of $\G_p\irestrict{X_0}\overset{\sigma}{\to}\X$. First assume that a local transformation $\p$ exists such that $$R_\p\cdot s=\sb\comp\p.$$ Taking the pull-back of a right-invariant Maurer-Cartan form $\muiK$, $|K|<p$, on $\G_p$ on both sides of this equation gives
\[
s^*R_\p^*\muiK=(\sb\comp\p)^*\muiK~~\iff~~s^*\muiK=\p^*\sb^*\muiK.
\]
This means that a necessary condition for there to exist an equivalence map $\p$ is that it preserves the pulled-back Maurer-Cartan forms $s^*\muiK$ and $\sb^*\muiK$. We shall prove the converse, i.e. that any local transformation, $f$, of $\X$ that preserves the set of pulled-back forms $s^*\muiK$ and $\sb^*\muiK$ must be a transformation from $\G$ and satisfy $R_f\cdot s=\sb$. Note that \emph{if} this result is indeed true and we are given a local transformation $f$ of $\X$ that preserves this collection of one-forms, we have
\[
R_{f(x)}\cdot s(x)=\sb(f(x))~~\iff~~s(x)\cdot j^pf^{-1}\resx=\sb(f(x)).
\]
We can solve for $j^pf\resx$ in this equation to obtain
\beq\label{jpf}
j^pf\resx=\sb(f(x))^{-1}\cdot s(x).
\eeq
Given a local transformation, $f$, of $\X$ that preserves $s^*\muiK$ and $\sb^*\muiK$, our method of proof will be to define a section of $\G_p$ by setting
\[
a(x):=\sb(f(x))^{-1}\cdot s(x),
\]
(note that this is indeed a section of $\G_p$) and proving that all the Maurer-Cartan forms $\muiK$, $|K|<p$, on $\G_p$ vanish when restricted to it. Since the Maurer-Cartan forms are a basis for the contact co-distribution on $\G_p$ this means that $j^pf\resx$ must indeed be the prolongation of a local transformation.

Equation (\ref{jpf}) motivates the definition of the map
\[
\mb:U_\mb\to\G_p,\quad \mb(g,h)=g^{-1}\cdot h,
\]
defined on the subset of $\G_p\times\G_p$ given by
\[
U_\mb=\{(g,h)\in\G_p\times\G_p~|~\tau(g)=\tau(h)\},
\]
i.e. on all pairs $(g,h)$ with a shared target. We need to know how the differential of this map behaves on $\tau$-vertical vectors. Denote the \emph{inverse map} on $\G_p$ by $\ib$. This map sends a groupoid element $j^p\p\resx$ to $j^p\p^{-1}\restrict{\p(x)}$ and is a diffeomorphism of $\G_p$. Notice that the differential $\ib_*$ maps vertical tangent vectors to $\tau$-vertical ones, and vice versa. We need the following two lemmas before giving our main result.

\begin{lemma}\label{m}
Let $V\in T_g\G_p\irestrict{\tau(g)}$ and $W\in T_h\G_p\irestrict{\tau(h)}$ be $\tau$-vertical vectors on $\G_p$ with $(g,h)\in U_\mb$. Then
\[
\mb_*(V, W)=R_{h^{-1}_*}\ib_*V + L_{g^{-1}_*}W.
\]
\end{lemma}

\begin{proof}
The $\tau$-vertical tangent vectors at each point of $\G_p$ form a vector space, and there must be some linear maps $A:T_g\G_p\irestrict{\tau(g)}\to T_{g^{-1}\cdot h}\G_p$ and $B:T_h\G_p\irestrict{\tau(h)}\to T_{g^{-1}\cdot h}\G_p$ such that
\[
\mb_*(V, W)=AV+BW.
\]
Now consider $\mb_*(V, 0)$, where $V\in T_g\G_p\irestrict{\tau(g)}$ and $0\in T_h\G_p\irestrict{\tau(h)}$, and let $\Phi(\ep)$ be a path in $\G_p$ with $\tau(\Phi(\ep))=\tau(g)=\tau(h)$ constant and with $\frac{d}{d\ep}\restrictbig{\ep=0}\Phi(\ep)=V$. Then
\begin{align*}
\mb_*(V, 0)&=\frac{d}{d\ep}\restrictbig{\ep=0}\mb(\Phi(\ep), h)=\frac{d}{d\ep}\restrictbig{\ep=0}\left(\Phi(\ep)^{-1}\cdot h\right)\\
&=\frac{d}{d\ep}\restrictbig{\ep=0}\left(R_{h^{-1}}\cdot \ib(\Phi(\ep))\right) = R_{h^{-1}_*}\ib_*V.
\end{align*}
This means that $A=R_{h^{-1}_*}\ib_*$. To find $B$, we choose a path $\Phi(\ep)$ with constant target $ \tau( \Phi(\varepsilon))=\tau(g)=\tau(h)$ such that $\frac{d}{d\ep}\restrictbig{\ep=0}\Phi(\ep)=W$ and compute (here $0\in T_g\G_p\irestrict{\tau(g)}$)
\[
\begin{aligned}
\mb_*(0, W)&=\frac{d}{d\ep}\restrictbig{\ep=0}\mb(g,\Phi(\ep))=\frac{d}{d\ep}\restrictbig{\ep=0}\left(g^{-1}\cdot\Phi(\ep)\right)\\
&=\frac{d}{d\ep}\restrictbig{\ep=0}\left(L_{g^{-1}}\cdot\Phi(\ep)\right)=L_{g^{-1}_*}W.
\end{aligned}
\]
This proves the lemma.
\end{proof} 

Denoting the identity section of $\G_p$ by $\mathbbm{1}$ we have the following.

\begin{lemma}\label{gg}
For a $\tau$-vertical $V\in T_g\G_p\irestrict{\tau(g)}$, we have
\[
\mb_*(V,V)=R_{g^{-1}_*}\ib_*V+L_{g^{-1}_*}V=\one_*\sigma_*V.
\]
\end{lemma}

\begin{proof}
Notice that $\mb(g,g)=g^{-1}\cdot g=\one(\sigma(g))$, and so, given a path $\Phi(\ep)$ with $\frac{d}{d\ep}\restrictbig{\ep=0}\Phi(\ep)=V$ we have
\[
\mb_*(V,V)=\frac{d}{d\ep}\restrictbig{\ep=0}\Phi(\ep)^{-1}\cdot \Phi(\ep)=\frac{d}{d\ep}\restrictbig{\ep=0}\one(\sigma(\Phi(\ep))=\one_*\sigma_*V.
\]
But by Lemma \ref{m} we also have
\[
\mb_*(V,V)=R_{g^{-1}_*}\ib_*V+L_{g^{-1}_*}V,
\]
proving the lemma.
\end{proof}

\begin{theorem}\label{Gco}
Let $s$ and $\sb$ be two sections of $\G_p\irestrict{X_0}\overset{\sigma}{\to}\X$ and let $f$ be a local transformation on $\X$ such that
\[
f^*\sb^*\muiK=s^*\muiK,\quad |K|<p.
\]
Then $f\in\G$ and $R_f\cdot s=\sb$.
\end{theorem}

\begin{proof}
Consider the section of $\G_p$ given by $$a(x)=\sb(f(x))^{-1}\cdot s(x)=\mb\comp(\sb(f(x)), s(x)).$$ Notice that this section agrees with $j^0f(x)$, i.e. the zero-jet of $f$. We shall prove that $a$ is the prolongation of a local transformation by showing $a^*\muiK=0$, for all $|K|<p$. Since $a$ agrees with $f$ at the zero order, we must have $a(x)=j^pf\resx$ and since (by definition of $a$) $R_{a(x)}\cdot s(x)=\sb(f(x))$, we have $R_f\cdot s=\sb$.

Let $v\in T_x\X$ be a tangent vector to $\X$ in the domain of $s$. We compute, using Lemma \ref{m},
\beq\label{a^*}
a^*\muiK(v)=\mb^*\muiK((\sb\comp f)_*v, s_*v)=\muiK(R_{s(x)^{-1}_*}\ib_*(\sb\comp f)_*v+L_{\sb(f(x))^{-1}_*}s_*v).
\eeq
First consider $\muiK(L_{\sb(f(x))^{-1}_*}s_*v)$. Since the determining equations for $\G_p$ are locally solvable, we can choose a local solution $\psi\in\G$ such that $j^p\psi\restrict{X_0}=\sb(f(x))^{-1}$ and because of Lemma \ref{same} we have
\beq\label{L_p}
\muiK(L_{\sb(f(x))^{-1}_*}s_*v)=\muiK(L_{\psi_*}s_*v)=(L_\psi^*\muiK)(s_*v).
\eeq
But $L_\psi^*\muiK$ is a \emph{right-invariant contact form} (of order $|K|$), since the left and right actions commute, and $L_\psi$ is a contact transformation on $\G_p$. Since the $\muiK$ are a basis for right-invariant contact forms on $\G_p$ and $s^*\muiK=(\sb\comp f)^*\muiK$ we also have $s^*(L_\psi^*\muiK)=(\sb\comp f)^*(L_\psi^*\muiK)$. Continuing (\ref{L_p}), and applying Lemma \ref{same} again, we get
\beq\label{sbf}
(L_\psi^*\muiK)(s_*v)=(L_\psi^*\muiK)((\sb\comp f)_*v)=\muiK(L_{\psi_*}(\sb\comp f)_*v)=\muiK(L_{\sb(f(x))^{-1}_*}(\sb\comp f)_*v).
\eeq
The last expression, according to Lemma \ref{gg}, is equal to
\[
\muiK(\one_*\sigma_*(\sb\comp f)_*v)-\muiK(R_{\sb(f(x))^{-1}_*}\ib_*(\sb\comp f)_*v)=-\muiK(R_{\sb(f(x))^{-1}_*}\ib_*(\sb\comp f)_*v),
\]
since $\one^*\muiK=0$, as $\one$ annihilates contact forms under pull-back. Plugging this into (\ref{a^*}) we arrive at
\beq\label{Ri}
a^*\muiK(v)=\muiK(R_{s(x)^{-1}_*}\ib_*(\sb\comp f)_*v)-\muiK(R_{(\sb(f(x))^{-1}_*}\ib_*(\sb\comp f)_*v).
\eeq
Now, as before, choose two local transformations $\p, \psi\in\G$ such that
\[
j^p\p\restrict{X_0}=s(x)^{-1},\quad j^p\psi\restrict{X_0}=\sb(f(x))^{-1}.
\]
According to Lemma \ref{same}, since $(\sb\comp f)_*v$ is $\tau$-vertical, we can write (\ref{Ri}) as
\[
\begin{aligned}
a^*\muiK(v)&=\muiK(R_{\p_*}\ib_*(\sb\comp f)_*v)-\muiK(R_{\psi_*}\ib_*(\sb\comp f)_*v)\\
&=\muiK(\ib_*(\sb\comp f)_*v)-\muiK(\ib_*(\sb\comp f)_*v)\\
&=0,
\end{aligned}
\]
since $R_\p^*\muiK=R_\psi^*\muiK=\muiK$.
\end{proof}

We easily obtain the following useful corollary.

\begin{corollary}\label{GcoCor}
Let $\G_p$ be formally integrable, let $X\in\X$ and let $\Lambda:\G_p\irestrict{X_0}\to\G_p\irestrict{X_0}$ be a map on a $\tau$-vertical fiber. Then $\Lambda$ preserves the Maurer-Cartan forms $\muiK$ (restricted to $\G_p$), $|K|<p$, if and only if $\Lambda=R_\p$ for some $\p\in \G$.
\end{corollary}

\begin{proof}
If $\Lambda=R_\p$ for $\p\in\G$ it obviously preserves the $\muiK$. Conversely, assume $\Lambda^*\muiK=\muiK$ for all $|K|<p$. Then $\Lambda$ is a diffeomorphism and, since $\mu^i=-\oi$ on $\G_p\irestrict{X_0}$, we have $\Lambda^*\oi=\oi$. Let $s$ be any section of $\G_p\irestrict{X_0}\overset{\sigma}{\to}\X$ and denote by $\sb$ the image of $\Lambda(s)$. Then $\sb$ is also a section of $\G_p\irestrict{X_0}$ since 
\[
\Lambda^*\left(\o^1\wedge\dots\wedge\o^n\restrict{\sb}\right)=\o^1\wedge\dots\wedge\o^n\restrict{s}\neq0,
\]
(where $\restrict{s}$ denotes \emph{pull-back} onto the \emph{image} of $s$ in $\G_p$) and hence $\o^1\wedge\dots\wedge\o^n\restrict{\sb}\neq0$. Then, according to Theorem \ref{Gco}, $\Lambda$ agrees with $R_\p$, for some $\p\in\G$, when restricted to $s$ and since the section $s$ was arbitrary, we also must have $\Lambda=R_\p$.
\end{proof}

\begin{Rem}\label{mueqom}
We mention that by the recurrence formula, on each $\tau$-vertical fiber $\G\irestrict{X_0}$, we have
\[
 \mu^i=-\oi,
\]
and in the forthcoming application of the above results we shall usually work with the $\oi$.
\end{Rem}

There is an easy generalization of Lemma \ref{m} available that will provide the solution to the general equivalence of general sections of $\G_p$. To describe it we define the operator $\mathbf{t}$ taking tangent vectors in $T\G_p$ to $\tau$-vertical ones by setting their target coordinate to zero. That is, define 
\[
 \mathbf{t}(V)=V-\tau_*V.
\]
Note that for any path $\ep\mapsto\Phi(\ep)$ in $U_\mb\subset\G_p\times\G_p$, through the point $\Phi(0)=(g,h)\in U_\mb$, there is a path $\ep\mapsto\gamma(\ep)$ in $\X$ such that
\[
\tau\times\tau\left(\Phi(\ep)\right)=(\gamma(\ep), \gamma(\ep)),
\]
with $\gamma(0)=\tau(g)=\tau(h)$. Defining a path $\widetilde{\Phi}(\ep)$ to have the same components as $\Phi(\ep)$ except with
\[
\tau\times\tau\left(\widetilde{\Phi}(\ep)\right)=(\gamma(0), \gamma(0))=(\tau(g), \tau(h))
\]
fixed, it is easy to see that the paths $\mb(\Phi(\ep))$ and $\mb(\widetilde{\Phi}(\ep))$ are the \emph{same}. Hence, for any $V\oplus W\in TU_\mb$, we have $\mb_*(V,W)=\mb(\mathbf{t}(V),\mathbf{t}(W))$ (where we identify $TU_\mb$ with a subspace of $T\G_p\times T\G_p$ in the obvious way). Lemma \ref{m} now directly gives the following.
\begin{lemma}
 Let $V\oplus W\in TU_\mb$ be a tangent vector at $(g,h)$. Then
 \[
  \mb_*(V,W)=\mb_*(\mathbf{t}(V),\mathbf{t}(W))=R_{h^{-1}_*}\ib_*\mathbf{t}(V)+ L_{g^{-1}_*}\mathbf{t}(W).
 \]

 \end{lemma}

Now, when two sections, $s$ and $\sb$, of $\G_p$ do not have fixed target coordinates, $\tau(s(x))=\tau(\sb(\xb))=X_0$, any equivalence map $R_\p$ between them must preserve the invariants
\[
I(x):=\tau(s(x))\quad\text{and}\quad\bar{I}(\xb):=\tau(\sb(\xb)).
\]
Conversely, given a local map $f:\X\to\X$ that preserves the pulled-back Maurer-Cartan forms 
\[
f^*\sb^*\muiK=s^*\muiK
\]
\emph{and} the invariants $I$ and $\bar{I}$,
\[
\bar{I}(f(x))=I(x),
\]
our proof of Theorem \ref{Gco} goes through essentially unchanged to prove that $f\in\G$ and $R_f$ maps $s$ to $\sb$; we just have to replace $(\sb\comp f)_*v$ and $s_*v$ by $\mathbf{t}\left((\sb\comp f)_*v\right)$ and $\mathbf{t}\left(s_*v\right)$ after the last equality in (\ref{a^*}) and throughout. Similarly we can prove the generalization of Corollary \ref{GcoCor}.

\begin{theorem}\label{Gco_gen}
Let $s$ and $\sb$ be two sections of $\G_p\to\X$ and let $f$ be a local diffeomorphism on $\X$ such that
\begin{align*}
f^*\sb^*\muiK&=s^*\muiK,\quad |K|<p,\quad\text{and}\\
f^*\sb^*\tau&=s^*\tau.
\end{align*}
Then $f\in\G$ and $R_f\cdot s=\sb$.
\end{theorem}

\begin{corollary}\label{Gco_gen_cor}
Let $\G_p$ be formally integrable, let $X\in\X$ and let $\Lambda:\G_p\to\G_p$. Then $\Lambda$ preserves the Maurer-Cartan forms $\muiK$ (restricted to $\G_p$), $|K|<p$, and the target map $\tau$ if and only if $\Lambda=R_\p$ for some $\p\in \G$.
\end{corollary}

\end{subsection}

\begin{subsection}{$G$-structures and partial moving frames}\label{partmaur}

Let a Lie pseudo-group $\G$ of local diffeomorphisms on $\X$ be determined by a set of first order equations (this is purely a simplifying, not a necessary, assumption). Assume the action of $\G$ is extended to a space $\U$ such that the action in the variables $u\in\U$ depends only on the first order jets of transformations from $\G$. As before, we ask when two local sections in $\E=\X\times\U\to\X$ are congruent under a transformation from $\G$. For our (local) purposes, it will be sufficient to consider sections whose images are the graphs of locally defined functions $u:\X\to\U$. For such a locally defined function $u$, we shall denote the section $x\mapsto (x,u(x))$ by $j^0u$ and its higher jets by $j^pu$, $0\leq p\leq\infty$. The equivalence problem for sections in $\E$ then asks, for two sections of $\E$, $j^0u$ and $j^0\ub$, if there is a $\p\in\G$ such that
\beq\label{ssb}
j^1\p\resx\cdot j^0u\resx=j^0\ub\restrict{\p(x)},
\eeq
for all $x$ in some open subset of $\text{dom}~u$, and how do we characterize all such congruences?
It will be convenient to have a running example during this section, but this next example also introduces the key idea. 

\bex\label{Qex1}
Let $\G$ be the Lie pseudo-group of contact-transformations in the variables $z=(x,u,u_x)=(x,u,p)$ determined by the differential equations
\beq\label{def1}
X_u=X_p=U_p=0,\quad P=\frac{U_x+pU_u}{X_x},\quad P_p=\frac{U_u}{X_x},
\eeq
extended to act on sections in $\R^3\times\R\to\R^3$, in coordinates $(x,u,p,q)=(z,q)$, by 
\[
j^0q\resz\mapsto \left(\p(z), Q(j^1\p\resz, j^0q\resz)\right),
\]
where the action on the $q$-variable is
\[
Q(j^1\p\resz, j^0q\resz)=\frac{P_x+pP_u+qP_p}{X_x}.
\]
Here, $j^0q$ is the section determined by the zero jet of a local function $q:\R^3\to\R$. Now, for the moment assume that the determining equations (\ref{def1}) for $\G$ are locally solvable (This can in fact be easily proven via Cartan-K\aa hler). Then Theorem \ref{Gco} says that two sections, $s$ and $\sb$, with fixed target coordinates, say $Z=(X,U,P)=0$, of $\G_1$ are congruent under the right action of a $\psi\in\G$ if and only if $\psi$ preserves the pull-backs of $\o^x$, $\o^u$ and $\o^p$ (Recall Remark \ref{mueqom}). In all applications we consider it will be sufficient to work with sections with fixed target coordinates, not needing the more general result of Theorem \ref{Gco_gen} and so we restrict this discussion to this case. This is just for simplification; all our results extend trivially to the more general \emph{non-transitive} case (but the equivalence map must then preserve the target coordinates $Z$). 

But what has to happen to guarantee that a local transformation $\psi\in\G$ maps a local section of $\E$ to another such section? Letting these sections be the jets of two locally defined functions $q_0, \bar{q}_0:\X\to\U$, we claim that this is the case if and only if \emph{there exists} a section $s$ of $\G_1$, with $R_\psi\cdot s=\sb$, such that 
\beq\label{Qj1}
Q(s\resz, j^0q_0\resz)=Q(\sb\restrict{\psi(z)}, j^0\bar{q}_0\restrict{\psi(z)}).
\eeq

To see why, notice that 
\begin{alignat*}{3}
{}&Q(s\resz\cdot j^1\psi^{-1}\restrict{\psi(z)}, Q(j^1\psi\resz, j^0q_0\resz))\hspace{30pt}&&\\
{}=&Q(s\resz, j^0q_0\resz) &&\text{by right invariance of $Q$}\\
{}=&Q(s\resz\cdot j^1\psi^{-1}\restrict{\psi(z)}, j^0\bar{q}_0\restrict{\psi(z)}) &&\text{by (\ref{Qj1})}.
\end{alignat*}
Comparing the first and last equations we can see that, if $Q(g,j^0q)$ is full rank in the $q$ variable, we have
\[
Q(j^1\psi\resz, j^0q\resz)=j^0\bar{q}\restrict{\psi(z)}.
\]
But this means that $\psi$ also maps $j^0q_0$ to $j^0\bar{q}_0$.
\eex

Note that, for the zero order lifted invariant $Q$, we have $Q(\one, j^0q)=q$ which is of maximal rank in $q$. Therefore, at least close to the identity section, $Q(g,j^0q)$ is full rank in $q$. This fact and the computations at the end of the last example easily generalize to prove the following.

\begin{theorem}\label{zfund}
Let $\G$ be a horizontal action on a trivial bundle $\E=\X\times\U\to\X$ whose $q\th$ order defining equations are formally integrable and assume the zero order lifted invariants $U=\lambda(u)$ are full rank in the pseudo-group parameters. Given two local functions, $u, \ub:\X\to\U$, there is a map $\p\in\G$ taking $j^0u$ to $j^0\ub$ if and only if there exists a sectiono of some target fiber $\G_q\irestrict{X_0}\overset{\sigma}{\to}\X$, $s$, such that $$U\restrictbig{(s\resx, j^0u\resx)}=U\restrictbig{(\sb\restrict{\psi(x)}, j^0\ub\restrict{\psi(x)})}.$$
\end{theorem}

Continuing Example \ref{Qex1}, we next introduce the fundamental idea underlying the moving frame.

\bex\label{Qex2}
We are given two local functions $q_0$ and $\qb_0$ and wish to characterize all equivalence maps, $\p\in\G$, between their graphs. We can do this, according to Theorem \ref{zfund}, by constructing two sections, $s$ and $\sb$, of $\G_1$ and requiring that $\p$, or, rather $R_\p$, maps $s$ to $\sb$ all the while preserving the lifted invariant 
\[
Q=\frac{P_x+pP_u+qP_p}{X_x},
\]
restricted to $(s,j^0q_0)$ and $(\sb, j^0\qb_0)$. But, and here is the key observation, since we get to choose the sections $s$ and $\sb$ ourselves, we might as well choose them to lie in the subset of $\G_1$ where $Q=0$. Then, any map between $s$ and $\sb$ \emph{automatically} preserves the restricted lifted invariant $Q$. Taking a closer look at this, we, for example, require that the components of $s$ satisfy
\beq\label{specu}
Q=\frac{P_x+pP_u+q_0P_p}{X_x}=0~~\iff~~P_x=-pP_u-q_0P_p,
\eeq
and similarly for $\sb$ and $\qb_0$. Referring back to our fundamental Theorem \ref{Gco}, the equivalence problem has now been reduced to finding a local transformation $f$ of $\X$, and two sections, $s$ and $\sb$, of the \emph{partial moving frame} determined by $X=U=P=Q=0$, such that
\[
f^*\sb^*\oi=s^*\oi.
\] 
The recurrence formula will hold all the relevant information of the structure of the Maurer-Cartan forms thereon. For example, on $Q=0$, we have
\[
0=dQ=\mu^p_x+Q_X\ox+Q_U\ou+Q_P\op.
\]

\eex

In order to capture the above ideas in cleaner notation we establish the following definition.

\begin{definition}\label{def:bu}
Let a partial moving frame (of order $p$) $\B_p\to\S$ have domain of definition $\S\subset J^r(\E)$ (for some $r$) and let $u$ be a local function $u:\X\to\U$ whose $r$ jet, $j^ru$, lands in $\S$. We denote the image, in $J^r(\E)$, of the section $j^ru$ by im($j^ru$) and denote by $\Bu_p\to\text{im}(j^ru)$ the restriction of $\B_p$ to source coordinates in im($j^ru$). Composing the bundle projection $\Bu_p\to\text{im}(j^ru)$ with the natural projection $\pi^r:J^r(\E)\to\X$ we will often identify $\Bu_p\to\text{im}(j^ru)$ with the resulting bundle $\Bu_p\to\X$. Note that $\Bu_p\to\X$ is a subbundle of $\G_p$.
\end{definition}

Generalizing the above example, we have the following corollary to Theorem \ref{zfund}.

\begin{corollary}\label{zfundcor}
Let the set-up be the same as in Theorem \ref{zfund} and let $\B_p\to\S$, for $p\geq q$, be a $p\th$ order partial moving frame with domain of definition $\S\subset J^r(\E)$ (for some $r\in\mathbb{N}$) on which we have $U=\lambda(u)=\text{constant}$. Then the graphs of two local functions, $u$ and $\ub$, such that im($j^ru$) and im($j^r\ub$) are in $\S$, are congruent under $\G$ if and only if there are two sections, $s$ of $\Bu_p\to\X$ and $\sb$ of $\Bub_p\to\X$, congruent under the right action of $\G$. 
\end{corollary}

Since $\Bu_p\to\X$ and $\Bub_p\to\X$ are subbundles of $\G_p$, the equivalence of sections, $s$ and $\sb$, of these bundles falls under the ambit of our fundamental Theorem \ref{Gco_gen}. But our results from last subsection provided a connection between the equivalence problem for sections in $\G_p$ and an equivalence problem of the Maurer-Cartan forms $\muiK$, $|K|<p$, on $\G_p$. This connection between partial moving frames (and the Maurer-Cartan forms restricted to these) and the equivalence problem for sections in $\E$ is the chief contribution of this paper!

\bre
When we normalize a lifted invariant the recurrence formula (\ref{rec}) tells us what effect that has on the structure of the Maurer-Cartan forms on the partial moving frame. We have, 
\[
0=d\Ua=\Ua_i\oi+\sum_{|K|\leq 1}\lambda\left(\frac{\partial \Ua}{\partial \XiK}\restrictbig{\one}\right)\muiK + \text{contact forms on}~\Jinf(\E).
\]
When we restrict to the (prolonged) graph of a specific function, $u$, the contact forms on $\Jinf(\E)$ vanish. Since an equivalence map between the graphs of $u$ and $\ub$ must preserve all the Maurer-Cartan forms restricted to the partial moving frame, we can see that such an equivalence map must also preserve the lifted invariants $\Ua_i$. We may normalize these to produce a new partial moving frame whose sections, restricted to $u$ and $\ub$, are equivalent if and only if $u$ and $\ub$ are. The next subsection continues this discussion. 
\end{Rem}

Cartan's equivalence method, and, indeed, most past approaches to Lie pseudo-groups have been based on the concept of a $G$-structure.

\begin{definition}\label{gstructure}
A $G$-structure for an equivalence problem of sections of $\E\to\X$ under the extended action of a Lie pseudo-group $\G\to\X$ is a pair $\{G,\pmb{\eta}\}$ where $G$ is a subgroup of the general linear group $GL(n)$ and $\pmb{\eta}=\{\eta^1,\ldots, \eta^n\}$ is a coframe of $\X$ (that depends on jets of sections of $\E$), that satisfies the following. A transformation $\p\in\G$ maps the graph of $u$ to the graph of $\ub$ if and only if $\pmb{\eta}$, restricted to the (prolongation of the) two graphs, is preserved, up to an element of $G\subset GL(n)$, under the pull-back of $\p$. That is
\beq\label{etaub}
\p^*\bbm \eta^1\restrict{\ub} \\ \vdots \\ \eta^n\restrict{\ub} \ebm=g\cdot  \bbm \eta^1\restrict{u} \\ \vdots \\ \eta^n\restrict{u} \ebm,
\eeq
for a map $x\mapsto g(x)\in G$. 
\end{definition}

\bre
We note that the condition (\ref{etaub}) is equivalent to there existing two locally defined functions, $s$ and $\sb$ from $\X$ to $G\subset GL(n)$ such that
\[
\p^*\left(\sb(\xb)\cdot\bbm \eta^1\restrict{\ub} \\ \vdots \\ \eta^n\restrict{\ub} \ebm\right) = s(x)\cdot \bbm \eta^1\restrict{u} \\ \vdots \\ \eta^n\restrict{u} \ebm.
\]
Given such $s$ and $\sb$ we can let $g=\sb(\p(x))^{-1}\cdot s(x)$ to recover (\ref{etaub}). The $G$-structure approach to equivalence problems then boils down to finding two equivalent sections, $s$ and $\sb$ of the trivial bundle $\X\times G\to\X$, the first restricted to $u$ and the second to $\ub$, just like in our partial moving frame approach.
\end{Rem}

\bex\label{xupex}
Consider the point transformation counterpart to our running example. That is, the Lie pseudo-group, $\G$, with defining equations $X_p=U_p=0$ and 
\beq\label{eq:Ppoint}
 P=\frac{U_x+pU_u}{X_x+pX_u} 
\eeq
acting on $q=u_{xx}$ by 
\[
q\mapsto\frac{P_x+pP_u+qP_q}{X_x+pX_u}.
\] 
Normalizing the lifted invariants $P$ and $Q$ to zero, we get $U_x=-pU_u$ and $P_x=-qP_p-pP_u$. Notice that differentiating the defining equation (\ref{eq:Ppoint}) with respect to $p$ gives the first order integrability condition $$pX_u+X_x=\frac{U_u-PX_u}{P_p},$$ which becomes $X_x+pX_u=\dfrac{U_u}{P_p}$ after setting $P=0$. Including this integrability condition makes $\G_1$ formally integrable. We have, on the partial moving frame $\B_1=\{X=U=P=Q=0\}$, that
\beq
\begin{aligned}
&\o^x=X_xdx+X_udu=(pX_u+X_x)dx+X_u(du-pdx)=\frac{U_u}{P_p}dx+X_u(du-pdx),\\
&\o^u=-pU_udx+U_udu=U_u(du-pdx),\\
&\o^p=(-qP_p-pP_u)dx+P_udu+P_pdp=P_u(du-pdx)+P_p(dp-qdx).
\end{aligned}
\eeq
If we further restrict to a specific function $q=f(x,u,p)$, we obtain the $G$ structure of Cartan for this problem, see \cite{olver95}: set $\eta^1=dx$, $\eta^2=du-pdx$ and $\eta^3=dp-fdx$ as a base coframe on $J^1$ and note that the lifted horizontal coframe can be written as
\beq\label{estr}
\bbm \o^x \\ \o^u \\ \o^p \ebm = \bbm \frac{U_u}{P_p} & X_u & 0 \\ 0 & U_u & 0 \\ 0 & P_u & P_p \ebm \bbm \eta^1 \\ \eta^2 \\ \eta^3 \ebm= \bbm \frac{a_1}{a_2} & a_3 & 0 \\ 0 & a_1 & 0 \\ 0 & a_4 & a_2 \ebm \bbm \eta^1 \\ \eta^2 \\ \eta^3 \ebm.
\eeq
By Theorem \ref{zfund} this is indeed a $G$-structure with $G\subset GL(3)$ being the subgroup of all matrices of the form
\[
\bbm \frac{a_1}{a_2} & a_3 & 0 \\ 0 & a_1 & 0 \\ 0 & a_4 & a_2 \ebm,\quad a_1a_2\neq0.
\]
\eex

\bex\label{el}
Understandably, the pseudo-group of contact transformations, $\G$, is prominent in equivalence problems for differential equations. One facet of which is the equivalence problem of Lagrangians $\dis\intc L(x,u,p)dx$ and $\dis\intc \bar{L}(X,U,P)dX$ (which, for simplicity, we take to be first order and in one variable). An element $\p\in\G$ transforms Lagrangians according to
\[
\intc \bar{L}(X,U,P)dX\mapsto \intc \bar{L}(X,U,P)\cdot(X_x+pX_u) dx,
\]
and we see that the equivalence problem can be written as the PDE
\[
\bar{L}(X,U,P)\cdot(X_x+pX_u)=L(x,u,p)~~\iff~~\bar{L}(X,U,P)=\frac{L(x,u,p)}{X_x+pX_u}.
\] 
If we extend $\G$ to act on a new variable $L$ by
\[
L\mapsto \frac{L}{X_x+pX_u},
\]
then the equivalence problem of Lagrangians is equivalent to that of sections in this extended space. Normalizing the lifted invariant $\displaystyle \frac{L}{X_x+pX_u}$ to 1, and, as before, $\displaystyle P=\frac{U_x+pU_u}{X_x+pX_u}$ to 0, we have $X_x+pX_u=L$, $U_x=-pU_u$ and recalling the integrability condition $\dis X_x+pX_u=\frac{U_u}{P_p}$ we have
\beq
\begin{aligned}
&\o^x=X_xdx+X_udu=(pX_u+X_x)dx+X_u(du-pdx)=Ldx+X_u(du-pdx),\\
&\o^u=-pU_udx+U_udu=U_u(du-pdx)=P_pL(du-pdx),\\
&\o^p=P_xdx+P_udu+P_pdp=\frac{P_x+pP_u}{L}Ldx+P_u(du-pdx)+P_pdp.
\end{aligned}
\eeq
Letting $a_1=X_u$, $a_2=P_p$, $\dis a_3=\frac{P_u}{L}$ and $\dis a_4=\frac{P_x+pP_u}{L}$ we recover a $G$-structure for this problem, with the base coframe $\{\eta^1, \eta^2, \eta^3\}=\{Ldx, L(du-pdx), dp\}$:
\[
\bbm \o^x \\ \o^u \\ \o^p \ebm = \bbm 1 & a_1 & 0 \\ 0 & a_2 & 0 \\ a_4 & a_3 & a_2 \ebm \bbm \eta^1 \\ \eta^2 \\ \eta^3 \ebm.
\]
\eex
 
 \bex\label{divel}
Sometimes a little work is required for the best possible formulation of an equivalence problem. Expanding on Example \ref{el}, we can consider the \emph{divergence} equivalence of first order Lagrangians under point transformations. It is well known that two Lagrangians produce the same Euler-Lagrange equations if they differ by a total derivative $D_xA(x,u)$, so the equivalence problem
\[
\bar{L}(X,U,P)\cdot (X_x+pX_u)=L(x,u,p)+D_xA,
\]
where $A(x,u)$ is some (real-analytic) function, is of interest. For our first order Lagrangian, $L$, the Euler-Lagrange expression has the form
\[
E(L):=L_u-L_{px}-pL_{pu}-qL_{pp},
\]
where $q$ is the second derivative of $u$. Under a contact transformation, the Euler-Lagrange equations transform according to
\[
E(L)\mapsto\frac{E(L)}{U_uX_x-U_xX_u}=\frac{E(L)}{P_p\cdot(X_x+pX_u)^2},
\]
where the second equality follows from differentiating both sides of $\dis P=\frac{U_x+pU_u}{X_x+pX_u}$ w.r.t. $p$. It will be convenient to define the truncated Euler-Lagrange expression $\widetilde{E}(L):=L_u-L_{px}-pL_{pu}$. Let us also denote the source coordinates $(x,u,p)$ by $t$ and the target coordinates $(X,U,P)$ by $T$. If two Lagrangians, $L$ and $\bar{L}$, are divergence equivalent, then their Euler-Lagrange equations are equivalent. This means that we must have
\[
E(\bar{L})(T)=\frac{E(L)(t)}{P_p\cdot(X_x+pX_u)^2}.
\]
Writing this out, and remembering that $q$ transforms according to $$q\mapsto Q=\frac{P_x+pP_u+qP_p}{X_x+pX_u},$$ we have
\[
\begin{aligned}
{}&\bar{L}_u(T)-\bar{L}_{px}(T)-P\bar{L}_{pu}(T)-\left(\frac{P_x+pP_u+qP_p}{X_x+pX_u}\right)\bar{L}_{pp}(T)\\
{}&=\frac{L_u(t)-L_{px}(t)-pL_{pu}(t)-qL_{pp}(t)}{P_p\cdot (X_x+pX_u)^2}.
\end{aligned}
\]
Viewing this as a first degree polynomial in $q$, and comparing coefficients we obtain
\[
\bar{L}_u(T)-\bar{L}_{px}(T)-P\bar{L}_{pu}(T)-\left(\frac{P_x+pP_u}{X_x+pX_u}\right)\bar{L}_{pp}(T)=\frac{L_u(t)-L_{px}(t)-pL_{pu}(t)}{P_p\cdot (X_x+pX_u)^2},
\]
and
\[
\frac{P_p}{X_x+pX_u}\bar{L}_{pp}(T)=\frac{L_{pp}(t)}{P_p\cdot (X_x+pX_u)^2}.
\]
Solving for $L_{pp}(T)$ in the second equation, plugging the result into the first and moving some terms around gives
\beq\label{pp}
\bar{L}_{pp}(T)=\frac{L_{pp}(t)}{P_p^2\cdot (X_x+pX_u)},
\eeq
and
\beq\label{etilde}
\widetilde{E}(\bar{L})(T)=\frac{\widetilde{E}(L)(t)}{P_p\cdot (X_x+pX_u)^2}+\frac{P_x+pP_u}{X_x+pX_u}\cdot\frac{L_{pp}(t)}{P_p^2\cdot(X_x+pX_u)}.
\eeq
The two left hand sides in (\ref{pp}) and (\ref{etilde}) only depend on target coordinates and are therefore invariant under the right action of $\G$. We can then extend the action of $\G$ to a bundle over the base manifold of the $(x,u,p)$, $\E$, parametrized by $(x,u,p,w,z)$ where $w=L_{pp}$ and $z=\widetilde{E}(L)$ and the action of $\G$ on these parameters is given by
\[
W=\frac{L_{pp}(t)}{P_p^2\cdot (X_x+pX_u)},
\]
and
\[
Z=\frac{\widetilde{E}(L)(t)}{P_p\cdot (X_x+pX_u)^2}+\frac{P_x+pP_u}{X_x+pX_u}\cdot\frac{L_{pp}(t)}{P_p^2\cdot(X_x+pX_u)}.
\] 
Two Lagrangians are divergence equivalent if and only if the corresponding sections they define in $\E$ are congruent under this extended action of $\G$ on $\E$. We can normalize both lifted invariants, $W$ and $Z$, setting the first to 1 and the second to zero. This gives
\[
X_x+pX_u=\frac{L_{pp}}{P_p^2},\quad P_x+pP_u=-\frac{\widetilde{E}(L)}{L_{pp}}P_p,
\]
where all sections are now being evaluated on the source coordinates $t=(x,u,p)$. Recalling the first order integrability condition $X_x+pX_u=\dfrac{U_u}{P_p}$ we also have on this partial moving frame that
\[
\frac{U_u}{P_p}=\frac{L_{pp}}{P_p^2}~~\Rightarrow~~U_u=\frac{L_{pp}}{P_p}
\]
and the lifted horizontal coframe (which is a basis for the restricted Maurer-Cartan forms on $\G_1$ to this first order partial moving frame) is
\beq
\begin{aligned}
&\o^x=X_xdx+X_udu=\frac{L_{pp}}{P_p^2}dx+X_u(du-pdx)=\frac{U_u^2}{L_{pp}}dx+X_u(du-pdx),\\
&\o^u=-pU_udx+U_udu=U_u(du-pdx),\\
&\o^p=P_xdx+P_udu+P_pdp=-\frac{\widetilde{E}(L)}{L_{pp}}P_pdx+P_u(du-pdx)+P_pdp\\
&~~~\hspace{0.5mm}=\frac{1}{U_u}\eta_c+P_u(du-pdx),
\end{aligned}
\eeq
where $\eta_c$ is the differential form
\[
\eta_c=-\widetilde{E}(L)dx+L_{pp}dp.
\]
We choose as group parameters
\[
a_1=X_u,~a_2=U_u,~a_3=P_u,
\]
and as a base coframe $\{\eta^1, \eta^2, \eta^3\}=\{\frac{1}{L_{pp}}dx, du-pdx,\eta_c\}$, but with these choices, we can write $\{\ox,\o^u, \o^p\}$ as
\[
\bbm \o^x \\ \o^u \\ \o^p \ebm = \bbm a_2^2 & a_1 & 0 \\ 0 & a_2 & 0 \\ 0 & a_3 & \frac{1}{a_2} \ebm \bbm \eta^1 \\ \eta^2 \\ \eta^3 \ebm.
\]
\eex

\end{subsection}

\end{section}

%% file: imf-Examples2.tex
\begin{section}{Examples}\label{examples1}

We now solve a few equivalence problems using our results. But first we offer a few words of guidance to the reader planning to apply this algorithm to an equivalence problem.

\begin{Rem}\label{guide}
The formulas for the lifted invariants $\UaJ$ quickly get out of hand as $|J|$ increases, and directly solving a system of normalization equations, $\UaJ=c\av_J$, will overwhelm computer algebra systems in most examples for $|J|\geq 3$. The key to effectively carry out the method is to take advantage of a few simplifying facts. 
\begin{itemize}
\item By far the most important fact is that the recurrence formula only involves the linearization of lifted invariants \emph{at the identity}. This provides immense simplification since we can obtain the form of the structure equations by computing
\[
\lambda(\vbinf(\uaj))=\frac{\partial \UaJ}{\partial X^i_K}\restrictbig{\one}\muiK,
\]
which is very easy for computer algebra systems at (essentially) arbitrary high orders of $|J|$ due to the prolongation formula. 
\item The linearized expressions
\[
\frac{\partial \UaJ}{\partial X^i_K}\restrictbig{\one}\muiK
\]
also tell us which group parameters can be normalized in $\UaJ$ \emph{without} us having to compute these normalizations.
\item As mentioned in Remark \ref{rem:sim} the fact that $\UaJ$ is affine in the top order group parameters helps in normalizing linear combinations of these. What the original wielders of the equivalence method did was essentially breaking each $\UaJ$ up into a top order part and lower order part and then solve a \emph{linear} system for the top order parts, never explicitly writing out the complicated expressions for lifted invariants. 
\item The structure equations of the Maurer-Cartan forms $\muiK$ and the recurrence formula will provide, with minimal effort, the structure equations of the various $G$-structures arising during the course of an equivalence problem. 
\end{itemize} 
\end{Rem}

Keeping these bullet points in mind, we now work out some examples.

\input{ex-2ndODEpoint}

\input{ex-secop}

\begin{subsection}{Medolaghi's pseudo-group}
\bex\label{medo}
Consider the Lie pseudo-group of transformations
\[
X=f(x),\quad Y=f_x(x)y+g(x),
\]
with defining equations
\[
Y_y=X_x,\quad X_y=0,
\]
extended to acting on the variable $u\in\R$ by
\[
u\mapsto U=u+ \frac{Y_x}{X_x}.
\]
As before, we are interested in the congruence problem of sections $(x,y)\mapsto (x,y,u(x,y))$ under this pseudo-group. We begin by normalizing $X=Y=U=0$, and the recurrence formula gives
\[
0=dU=U_X\o^x+U_Y\o^y + \mu^y_x-\mu^x_x.
\]
We normalize $Y_x$ from $U=0$ to obtain $Y_x=-uX_x$ and the zero-order structure equations become
\beq
\begin{aligned}
{}&d\o^x=\o^x\wedge\mu^x_x,\\
&d\o^y=\o^x\wedge\mu^y_x+\o^y\wedge\mu^y_y=\o^x\wedge\mu^x_x-U_Y\o^y\wedge\o^x+\o^y\wedge\mu^x_x.
\end{aligned}
\eeq
Taking the exterior derivative on both sides indicates that, modulo lower order Maurer-Cartan forms, $d_GU_Y=-\mu^x_{xx}$. We can therefore normalize $U_{Y}=0$ and solve for $X_{xx}$, sparing the reader the details, we obtain $X_{xx}=-u_yX_x$. The recurrence formula is $0=dU_Y=U_{Yi}\oi-\mu^x_{xx}$. The lifted invariant $U_X$ vanishes along with $\ox\wedge \ox$ in the structure equations, but this indicates that $Y_{xx}$ can be normalized in $U_X$. There are no non-normalized second order group parameters left so we have complete reduction on a first order partial moving frame $\B_1$. We have only one non-normalized group parameter of order one, and we compute
\[
d\mu^x_x=\o^x\wedge\mu^x_{xx}=U_{YY}\o^x\wedge\o^y.
\]
One consequence of this formula is that $U_{YY}$ must depend on \emph{first} order group parameters only. Indeed, on our restricted space, we have
\[
U_{YY}=\frac{u_{yy}}{X_x^2}.
\]
Here is where our equivalence problem branches. If $u_{yy}<0$ we can normalize $U_{YY}=-1$, but if $u_{yy}=0$ we can not normalize $U_{YY}$ at all. Focusing on the third branch, where $u_{yy}>0$, we normalize $U_{YY}$ to 1 and obtain the deterministic coframe
\beq\label{uyy}
\begin{aligned}
{}&\o^x=\sqrt{u_{yy}}dx,\\
&\o^y=-u\sqrt{u_{yy}}dx+\sqrt{u_{yy}}dy.
\end{aligned}
\eeq
The only genuine invariants of this problem will appear as structure functions of this coframe on the base manifold and their coframe derivatives. The recurrence formula gives $\displaystyle 0=dU_{YY}=U_{YYi}\oi-2\mu^x_x$ and so
\beq
\begin{aligned}
{}&d\o^x=\frac{1}{2}U_{YYY}\o^x\wedge\o^y,\\
&d\o^y=\frac{1}{2}(U_{YYY}-U_{YYX})\o^x\wedge\o^y.
\end{aligned}
\eeq
We can compute the invariants $U_{YYY}$ and $U_{YYX}$ either directly from (\ref{uyy}) or by first computing their lifted form and plugging in our normalizations. In any case, we have
\[
U_{YYY}=\frac{u_{yyy}}{u^{3/2}_{yy}},\quad U_{YYX}=\frac{u_{xyy}+uu_{yyy}+2u_yu_{yy}}{u^{3/2}_{yy}}.
\]

\begin{Rem}
Notice, again, that our methods have lead naturally to a generating system of invariants, namely $U_{YYY}$ and $U_{YYX}$ restricted to our moving frame.
\end{Rem}

On sections that are affine in $y$, i.e. $u_{yy}=0$, we have $d\mu^x_x=0$ and the Poincar\'e-lemma guarantees that there exists a function $\alpha$ such that $d\alpha=\mu^x_x$. Using the definition of the unrestricted Maurer-Cartan form, $$\mu^x_x=\frac{1}{X_x}\left(dX_x-X_{xx}dx\right),$$ we see that on our partial moving frame, where we normalized $X_{xx}=-u_yX_x$, we have $\displaystyle \mu^x_x=\frac{dX_x}{X_x}+u_ydx=d\left(\log(X_x)+xu_y\right)$ and the function $\alpha$ must have the form
\[
\alpha=\log(X_x)+\intc u_ydx.
\] 
The coframe can now be written
\beq
\begin{aligned}
{}&\o^x=\frac{e^\alpha}{e^{xu_y}} dx,\\
&\o^y=-\frac{ue^\alpha}{e^{xu_y}} dx+\frac{e^\alpha}{e^{xu_y}} dy,\\
&\sigma=d\alpha,
\end{aligned}
\eeq
The structure functions in this case are
\beq\label{alphaox}
\begin{aligned}
{}&d\o^x=\o^x\wedge\sigma,\\
&d\o^y=\o^x\wedge\sigma+\o^y\wedge\sigma,\\
&d\sigma=0.
\end{aligned}
\eeq
All the coefficients in these structure functions are constant so, in particular, the symmetry group of any affine section is a finite dimensional (local) Lie group with Lie algebra structure (\ref{alphaox}).

\begin{Rem}\label{singular_ex}
Notice that sections satisfying $u_{yy}=0$ do not make up a locally $\G$-invariant set $\S$ in $\Jinf(\E)$. Further, the action of $\G^\S$ on $\S$ (cf. Definition \ref{parmov}) will never be free as $X_x$ cannot be normalized and therefore the equivariant moving frame is, on its own, not able to deduce the symmetry properties of sections $\jinf u\in\S$. Since our formulation, via Theorem \ref{Gco}, characterizes $\G^\S$ by the collection of maps preserving restricted Maurer-Cartan forms these ``singular'' sections are placed on an entirely equal footing as regular jets and their analysis is no different. 
\end{Rem}

\eex
\end{subsection}

\end{section}

%% file: ex-2ndODEpoint.tex
\begin{subsection}{Equivalence of second order ordinary differential equations}
\begin{Ex}\label{point}
Consider the pseudo-group $\G$, acting on $\R^3$ with coordinates $(x,u,u_x)=(x,u,p)$, with defining equations
\[
X_p=U_p=0,\quad P=\frac{U_x+pU_u}{X_x+pX_u}.
\]
This is the pseudo-group from Example \ref{xupex} of point transformations. We spot the obvious integrability condition $P_p=(U_u-PX_u)/(X_x+pX_u)$ and the corresponding dependency $\mu^p_p=\mu^u_u-\mu^x_x$, and notice it is the only integrability condition appearing at the first stage, and adding it to $\G_1$ in fact makes $\G_1$ formally integrable. The other dependencies among the Maurer-Cartan forms on $\G_1$ are
\beq\label{xupP}
\mu^x_p=\mu^u_p=0,\quad \mu^p=\mu^u_x+P\mu^u_u-P\mu^x_x-P^2\mu^x_u.
\eeq
We extend the action of $\G$ on itself to the extra coordinate $u_{xx}=q$,
\[
q\mapsto Q=\frac{-q P X_u+q U_u+p^2 P_u X_u+p P_x X_u+p
   P_u X_x+P_x X_x}{\left(p
   X_u+X_x\right)^2},
\]
and consider the equivalence problem for sections $q$, $(x,u,p)\mapsto(x,u,p,q(x,u,p))$, i.e. second order ordinary differential equations. We normalize all zero order lifted invariants, $X=U=P=Q=0$, resulting, first of all in
\[
\ox=-\mu^x,~~\o^u=-\mu^y,~~\o^p=-\mu^p=\mu^u_x.
\]
By the recurrence formula, we have
\begin{align*}
0&=dQ=Q_X\ox+Q_U\o^u+Q_P\o^p + \mu^p_x+P\mu^p_u+Q\mu^u_u-2Q\mu^x_x-3QP\mu^x_u\\
&=Q_X\ox+Q_U\o^u+Q_P\o^p + \mu^p_x.
\end{align*}

After normalizing $X=U=P=Q=0$ we are working on a partial moving frame $\B_1$ (as always we do not restrict to a specific section, working instead with the general jet coordinates $j^\infty q\resx=(x,q,\ldots, q_J, \ldots)$). Next we compute, on $\Bt_1$, using all the linear dependencies among the $\mu$ from above,

\beq\label{point_str1}
\begin{aligned}
d\o^x&=\ox\wedge\mu^x_x+\o^u\wedge\mu^x_u,\\
d\o^u&=\ox\wedge\o^p+\o^u\wedge\mu^u_u,\\
d\o^p&=(Q_X\ox+Q_U\o^u+Q_P\o^p)\wedge\ox+\o^u\wedge\mu^p_u+\o^p\wedge(\mu^u_u-\mu^x_x).
\end{aligned}
\eeq
The horizontal parts we need to normalize are $Q_U\o^u\wedge\o^x$ and $Q_P\o^p\wedge\o^x$. The lifted invariant $Q_X$ disappears due to $\ox\wedge\ox=0$ and we can assume it is normalized to zero also; $Q_X$ depends on $P_{xx}$ which does not appear in any other lifted invariant and cannot contribute to any integral conditions/lower order invariants.

The recurrence formula gives
\begin{align*}
d_GQ_X&=\mu^p_{xx}-Q_P \mu^p_x+Q_X \mu^u_u-3Q_X\mu^x_x,\\
d_GQ_U&=\mu^p_{ux}-Q_P\mu^p_u-2 Q_U\mu^x_x,\\
d_GQ_P&=-\mu^x_{xx}-Q_P\mu^x_x+2\mu^p_u,
\end{align*}
and we normalize these lifted invariants to zero. The structure equations (\ref{point_str1}) become
\beq\label{point_str2}
\begin{aligned}
d\o^x&=\ox\wedge\mu^x_x+\o^u\wedge\mu^x_u,\\
d\o^u&=\ox\wedge\o^p+\o^u\wedge\mu^u_u,\\
d\o^p&=\o^u\wedge\mu^p_u+\o^p\wedge(\mu^u_u-\mu^x_x)
\end{aligned}
\eeq
and we now test for involution by looking at the set
\[
\gamma_1\{\left(a\pder{\o^x}+b\pder{\o^u}+c\pder{\o^p}\right)\interior d\oi~|~i\in\{x,u,p\}\},
\]
where $\gamma_1$ projects onto the space of first order Maurer-Cartan forms. This set is
\[
\{a\mu^x_x+b\mu^x_u,~~ b\mu^u_u,~~ b\mu^p_u+c\mu^u_u-c\mu^x_x\}
\]
and maximizing its rank is equivalent to maximizing the rank of the matrix
\[
\bbm a & b & 0 & 0 
\\
0 & 0 & b & 0
\\
-c & 0 & c & b
\ebm.
\] 
We find the reduced Cartan characters $s^{(1)}_1=3$, $s^{(1)}_2=1$ and $s^{(1)}_3=0$. The second order pseudo-group parameters we have not yet normalized are $\{P_{uu}, U_{uu}, X_{uu}, X_{ux}\}$ and Cartan's test ask whether
\[
4=\#\{P_{uu}, U_{uu}, X_{uu}, X_{ux}\}=\sum is^{(1)}_i=1\cdot3+2\cdot 2+0\cdot 3=5.
\]
This is untrue and the equivalence problem is not in involution at this point. We must therefore prolong to include $\mu^x_x, \mu^x_u, \mu^u_u$ and $\mu^p_u$. The structure equations are (\ref{point_str2}) and
\beq\label{point_str3}
\begin{aligned}
d\mu^x_x&=\ox\wedge(Q_{Pi}\oi+2\mu^p_u)+\o^u\wedge\mu^u_{ux}+\mu^x_u\wedge\o^p,\\
d\mu^x_u&=\ox\wedge\mu^x_{ux}+\o^u\wedge\mu^x_{uu}+\mu^x_x\wedge\mu^x_u+\mu^x_u\wedge\o^p,\\
d\mu^u_u&=\o^x\wedge\mu^p_u+\o^u\wedge\mu^u_{uu}+\o^p\wedge\mu^x_u,\\
d\mu^p_u&=\o^x\wedge(-Q_{Ui}\oi)+\o^u\wedge\mu^p_{uu}+\o^p\wedge(\mu^u_{uu}-\mu^x_{ux})-\mu^x_x\wedge\mu^p_u.
\end{aligned}
\eeq
We now turn to the second order invariants using the recurrence formula but we find that
\begin{align*}
dQ_{PP}&=Q_{PPi}\oi+2\mu^u_{uu}-4\mu^x_{ux},\\
d_GQ_{PPP}&=Q_{PPPi}\oi-6\mu^x_{uu}.
\end{align*}
and so we can normalize $U_{uu}$ and $X_{uu}$ and reduce our structure equations according to
\begin{align*}
\mu^u_{uu}&=-\frac{1}{2}Q_{PPi}\oi+2\mu^x_{ux},
\\
\mu^x_{uu}&=\frac{1}{6}Q_{PPPi}\oi.
\end{align*}
The equations (\ref{point_str3}) become
\beq\label{point_str4}
\begin{aligned}
d\mu^x_x&=\ox\wedge(Q_{Pi}\oi+2\mu^p_u)+\o^u\wedge\mu^u_{ux}+\mu^x_u\wedge\o^p,\\
d\mu^x_u&=\ox\wedge\mu^x_{ux}+\o^u\wedge(\frac{1}{6}Q_{PPPi}\oi)+\mu^x_x\wedge\mu^x_u+\mu^x_u\wedge\o^p,\\
d\mu^u_u&=\o^x\wedge\mu^p_u+\o^u\wedge(-\frac{1}{2}Q_{PPi}\oi+2\mu^x_{ux})+\o^p\wedge\mu^x_u,\\
d\mu^p_u&=\o^x\wedge(-Q_{Ui}\oi)+\o^u\wedge\mu^p_{uu}+\o^p\wedge(-\frac{1}{2}Q_{PPi}\oi+\mu^x_{ux})-\mu^x_x\wedge\mu^p_u.
\end{aligned}
\eeq
Using the recurrence formula we also find
\begin{align*}
d_GQ_{PU}&=-\mu^x_{uxx}+2\mu^p_{uu},\\
d_GQ_{PX}&=-\mu^x_{xxx}+2\mu^p_{ux},\\
d_GQ_{UU}&=\mu^p_{uux},\\
d_GQ_{UX}&=\mu^p_{uxx},\\
d_GQ_{XX}&=\mu^p_{xxx},\\
d_GQ_{PPU}&=2\mu^u_{uuu}-4\mu^x_{uux}-Q_{PPX}\mu^u_x,\\
d_GQ_{PPX}&=-4\mu^x_{uxx}+2\mu^p_{uu}-Q_{PPX}\mu^x_x-Q_{PPX}\mu^u_u,\\
d_GQ_{PPPX}&=-6\mu^x_{uux},\\
d_GQ_{PPPU}&=-6\mu^x_{uuu}-Q_{4P}\mu^p_{u}-Q_{PPPX}\mu^x_u,\\
d_GQ_{4P}&=2Q_{4P}\mu^x_x-3Q_{4P}\mu^u_u.
\end{align*}
The equations for $Q_{PU}$ and $Q_{PPX}$ imply that if we normalize $X_{uxx}$ from $Q_{PU}$ then $dQ_{PPX}$ becomes
\[
-6\mu^p_{uu}+(Q_{PPXi}-4Q_{PUi})\oi-Q_{PPX}\mu^x_x-Q_{PPX}\mu^u_u
\]
and $P_{uu}$ can then be normalized from $Q_{PPX}=0$, in which case
\[
\mu^p_{uu}=\frac{1}{6}(Q_{PPXi}-4Q_{PUi})\oi.
\]
Normalizing in this manner reduces the structure equations (\ref{point_str4}) to (writing $I:=Q_{4P}$)
\beq\label{point_str5}
\begin{aligned}
d\mu^x_x&=2\ox\wedge\mu^p_u+\o^u\wedge\mu^u_{ux}+\mu^x_u\wedge\o^p,\\
d\mu^x_u&=\ox\wedge\mu^x_{ux}+\frac{1}{6}I\o^u\wedge\o^p+\mu^x_x\wedge\mu^x_u+\mu^x_u\wedge\o^p,\\
d\mu^u_u&=\o^x\wedge\mu^p_u+2\o^u\wedge\mu^x_{ux}+\o^p\wedge\mu^x_u,\\
d\mu^p_u&=\frac{1}{6}\o^u\wedge\left((Q_{PPXi}-4Q_{PUi})\oi\right)+\o^p\wedge\mu^x_{ux}-\mu^x_x\wedge\mu^p_u.
\end{aligned}
\eeq
We continue our toil of studying the $d_G$ of the lifted invariants appearing in the above equations and the recurrence formula, once again, gives
\[
d_G(Q_{PPXU}-4Q_{PUU})=-6\mu^p_{uuu}-(Q_{PPXX}-4Q_{PUX})\mu^x_u.
\]
We also find that
\[
d_G(Q_{PPXX}-4Q_{PUX})=-6\mu^p_{uux}-2(Q_{PPXX}-4Q_{PUX})\mu^x_x-(Q_{PPXX}-4Q_{PUX})\mu^u_u
\]
and according to the normalization of $Q_{UU}$ we find that, on our partial moving frame (writing $J:=Q_{PPXX}-4Q_{PUX}$),
\[
d_GJ=-2J\mu^x_x-J\mu^u_u
\]
and so we will possibly be able to normalize $X_x$ or $U_u$ from $J$, depending on the precise form of $J$.

At this point we find that we have normalized \emph{all} third order pseudo-group parameters and this equivalence problem will always reduce to a standard equivalence problem for \emph{coframes}. The space on which this coframe lives is parametrized by $x,u,p,X_x,X_u, U_u, P_u$ and $X_{ux}$, but this last coordinate is the only second order pseudo-group parameter we have not normalized. In addition to the above structure equations we join the equation for $d\mu^x_{ux}$ which we can read off the diffeomorphism pseudo-group structure equations (\ref{strrem}),
\[
d\mu^x_{ux}=\frac{1}{6}\o^x\wedge\left((Q_{PPXi}+2Q_{PUi})\right)\oi+\frac{1}{6}\o^u\wedge(Q_{PPPXi}\oi)+\mu^x_x\wedge\mu^x_{ux}+\mu^x_u\wedge\mu^p_{u}+\mu^x_{ux}\wedge\mu^u_u+\mu^x_{ux}\wedge\mu^x_u.
\]
The non-constant coefficients in the above equation are $K:=Q_{PPUX}+2Q_{PUU}-Q_{PPPXX}$ and $L:=Q_{PPPPX}$. We now must compute $I, J, K$ and $L$. We find that
\[
I=\frac{X_x+pX_u}{U_u^3}q_{pppp},
\]
while $J$, $K$ and $L$ have more complicated expressions. For example,
\begin{footnotesize}
\begin{align*}
J\cdot U_u
   \left(p X_u+X_x\right)^2&=q^2 q_{pppp}+p^2 q_{ppuu}-4 q_{pux}+2 p
   q_{ppux}-3 q q_{ppu}+2q p q_{pppu}-3 q_u
   q_{pp}+4 q_p q_{pu}+p q_u q_{ppp}\\
   &-p q_p
   q_{ppu}-4 p q_{puu}+2 q q_{pppx}+q_x
   q_{ppp}-q_p q_{ppx}+q_{ppxx}+6 q_{uu}.
\end{align*}
\end{footnotesize}
The final structure equations for our coframe are
\begin{align*}
d\o^x&=\ox\wedge\mu^x_x+\o^u\wedge\mu^x_u,\\
d\o^u&=\ox\wedge\o^p+\o^u\wedge\mu^u_u,\\
d\o^p&=\o^u\wedge\mu^p_u+\o^p\wedge(\mu^u_u-\mu^x_x),\\
d\mu^x_x&=2\ox\wedge\mu^p_u+\o^u\wedge\mu^u_{ux}+\mu^x_u\wedge\o^p,\\
d\mu^x_u&=\ox\wedge\mu^x_{ux}+\frac{1}{6}I\o^u\wedge\o^p+\mu^x_x\wedge\mu^x_u+\mu^x_u\wedge\o^p,\\
d\mu^u_u&=\o^x\wedge\mu^p_u+2\o^u\wedge\mu^x_{ux}+\o^p\wedge\mu^x_u,\\
d\mu^p_u&=\frac{1}{6}J\o^u\wedge\o^x+\o^p\wedge\mu^x_{ux}-\mu^x_x\wedge\mu^p_u,\\
d\mu^x_{ux}&=\frac{1}{6}K\o^x\wedge\o^u+\frac{1}{6}L\o^u\wedge\o^p+\mu^x_u\wedge\mu^p_{u}+\mu^x_{ux}\wedge\mu^u_u+\mu^x_{ux}\wedge\mu^x_u.
\end{align*}
The equivalence procedure now continues with a detailed study of the invariants $I, J, K$ and $L$. The complete analysis of this equivalence problem would take up another couple of dozen pages, but this analysis for recently completed for the first time, cf. \cite{milson15}.

\eex
\end{subsection}

%% file: ex-secop.tex
\begin{subsection}{Equivalence of differential operators}

\bex\label{ex:diffop}
Consider a linear second order differential operator on $\R$,
\beq\label{eq:D}
\mathscr{D}=fD^2+gD+h,
\eeq
where $f,g,h:\R\to\R$ are real-analytic, and $f\neq0$. When we apply $\mathscr{D}$ to a real-analytic function $u:\R\to\R$ we obtain the function
\[
fu''+gu'+hu.
\]
Now consider the pseudo-group, $\G$, of transformations of the $(x,u)$ of the form
\[
(x,u)\mapsto (\p(x), u\cdot\psi(x))=(X, U),
\]
where $\p, \psi:\R\to\R$ are real-analytic. Restricting to the set $\X=\{(x,u)\in\R^2~|~u>0\}$ and to $\psi>0$ this is a Lie pseudo-group of transformations of $\E$ that has defining equations
\[
X_u=0,\quad U_{uu}=0,\quad U_{u}u=U.
\]
The elements of $\G$ preserve the space of linear operators and a transformation from $\G$ maps (\ref{eq:D}) to
\[
\bar{\mathscr{D}}=F\bar{D}^2+G\bar{D}+H,
\]
where $\bar{D}$ is the derivative with respect to the transformed independent variables $X$ and the lifted coefficients $F, G, H$ have explicit formulas
\beq\label{eq:FGH}
\begin{aligned}
F&=f\frac{X_x^2}{U_u},\\
G&=-f\frac{2 U_xX_x-X_{xx}U_uu}{uU_u^2}+g\frac{X_x}{U_u},\\
H&=-f\frac{U_{xx}U_uu-2U_x^2}{u^2 U_u^3}-g\frac{U_x}{u U_u^2}+\frac{h}{U_u}.
\end{aligned}
\eeq
Notice that since $\frac{U}{u}$ is independent of $u$, these lifted invariants are also independent of $u$ (as they should be be). Each operator $\mathscr{D}$ defines a section,
\[
(x,u)\mapsto (x,u, f(x), g(x), h(x)),
\]
in the trivial bundle $\X\times\R^3\to\X$ and two operators are equivalent if their respective sections are congruent under the extended action of $\G$ given by (\ref{eq:FGH}).

Notice that the extended action does depend on the second order group parameters and so we shall, initially, be working in $\Gt_2$. The recurrence formula gives
\begin{align*}
0&=dF=F_X\ox+\o^u+2\mu^x_x,\\
0&=dG=G_X\ox-2\mu^u_x+\mu^x_{xx},\\
0&=dH=H_X\ox-\mu^u_{xx}.
\end{align*}

We normalize $F=1$ and $G=H=0$ and, in particular, obtain the normalization
\[
X_x=\frac{1}{\sqrt{fu}}.
\]
The structure equations are
\beq\label{ostr1}
\begin{aligned}
d\o^x&=\o^x\wedge\mu^x_x=-\frac{1}{2}F_{X}\ox\wedge\ox=0\\
d\o^u&=\o^x\wedge\mu^u_x.
\end{aligned}
\eeq
We have in fact normalized all second order pseudo-group parameters and so we have a fully determinate equivalence problem for the coframe $\{\o^x, \o^u, \mu^u_x\}$. The final structure equation is
\[
d\mu^u_x=F_X\o^x\wedge\mu^u_x+\o^u\wedge\mu^u_x.
\]
We compute 
\[
dF_X=(F_{XX}-2G_X)\ox+3\mu^u_x
\]
and realize that $U_x$ can be normalized from $F_X=0$. Indeed
\[
F_X=0\quad \iff \quad U_x=\frac{2g-f'}{3f}
\]
and we have normalized all pseudo-group parameters and so the equivalence problem has been reduced to an equivalence problem of coframes and invariants on $\X=\{(x,u)\in\R^2~|~u>0\}$. Now, on this final moving frame $\mu^u_x$ is reduced to  
\[
\mu^u_x=-\frac{1}{3}(F_{XX}-2G_X)\ox
\]
and so an equivalence map must preserve the coframe
\beq\label{cofr}
\begin{aligned}
\o^x&=\frac{1}{\sqrt{fu}}dx,\\
\o^u&=\frac{2g-f'}{3f}dx+\frac{1}{u}du,
\end{aligned}
\eeq
with structure equations $d\ox=d\o^u=0$ and the invariant 
\[
F_{XX}-2G_X=\frac{u \left(f f''+2 g f'-3 f g'+5 f h-g^2\right)}{f}.
\]
In particular, if $f f''+2 g f'-3 f g'+5 f h-g^2=0$ the differential operator has symmetry group $\R^2$, and all such operators are equivalent. The canonical operator in this class can be taken as $\mathscr{D}=D^2$ which is easily seen to have symmetry group $\R_+\times\R \cong \R^2$ acting via
\[
(x,u)\mapsto (\lambda x + \ep, \lambda^2 u)
\]
for $\lambda>0$, $\ep\in\R$.

\eex
\end{subsection}